\documentclass[11pt,a4paper]{article}

\usepackage[T1]{fontenc}
\usepackage{lmodern}
\usepackage[utf8]{inputenc}

\usepackage{tikz}
\usepackage{pgfplots}
\pgfplotsset{compat=newest}

\usepackage[left=2.0cm, top=2.45cm,bottom=2.45cm,right=2.0cm]{geometry}
\usepackage{mathtools,amssymb,amsthm,mathrsfs,calc,graphicx,x
r,dsfont,tikz,bm}
\usepackage[british]{babel}
\usepackage{amsfonts}              % for blackboard bold, etc

\usepackage[bbgreekl]{mathbbol}
\usepackage{bbm}
\DeclareSymbolFontAlphabet{\mathbb}{AMSb}

\usepackage[
  bookmarks=true,
  bookmarksnumbered=true,
  bookmarksopen=true,
  unicode=true,
  pdftoolbar=true,
  pdfmenubar=true,
  pdffitwindow=false,
  pdfstartview={FitH},
  pdftitle={},
  pdfauthor={},
  pdfsubject={},
  pdfcreator={},
  pdfproducer={},
  pdfkeywords={},
  pdfnewwindow=true,
  colorlinks=true,
  linkcolor=black,
  citecolor=black,
  filecolor=black,
urlcolor=black]{hyperref}

\usepackage{cleveref}
\usepackage[colorinlistoftodos]{todonotes}

\usepackage[T1]{fontenc}
\usepackage[deletedmarkup=xout]{changes}
\usepackage{selinput}
\numberwithin{equation}{section}
\makeatletter
\def\@cite#1#2{[\textbf{#1}\if@tempswa , #2\fi]}  %text
\def\@biblabel#1{[#1]}                %bibliography
\makeatother

\newtheorem {theorem}{Theorem}[section]
\newtheorem {proposition}[theorem]{Proposition}
\newtheorem {lemma}[theorem]{Lemma}
\newtheorem {corollary}[theorem]{Corollary}

\newtheorem {remark}[theorem]{Remark}

\newcommand{\var}{\operatorname{var}}

\newcommand{\dint}{\textup{d}}

\def\EE{\mathbb{E}}
\def\NN{\mathbb{N}}
\def\CC{\mathbb{C}}
\def\PP{\mathbb{P}}
\def\RR{\mathbb{R}}

\def\ZZ{\mathbb{Z}}

\def\bX{\mathbf{X}}

\def\bv{\mathbf{v}}

\def\arsinh{\mathrm{arsinh}}

\newcommand{\vol}{\operatorname{Vol}}
\newcommand{\conv}{\operatorname{conv}}

\makeatletter
\let\@fnsymbol\@alph
\makeatother

\begin{document}

\title{\bfseries Random convex chains through the lens\\ of analytic combinatorics}

\author{Florian Besau\footnotemark[1]\;\;\;and Christoph Th\"ale\footnotemark[2]}

\renewcommand{\thefootnote}{\fnsymbol{footnote}}

\footnotetext[1]{
  Friedrich-Schiller-Universität Jena, Germany. Email: florian@besau.xyz
}
\footnotetext[2]{%
  Ruhr-Universität Bochum, Germany. Email: christoph.thaele@rub.de
}

\date{}

\maketitle

\begin{abstract}
Consider the triangle $T$ with vertices $(0,0)$, $(0,1)$, and $(1,0)$. 
The lower boundary of the convex hull of $(0,1)$, $(1,0)$, together with $n$ independent uniformly distributed random points in $T$, is called a random convex chain and denoted by $T_n$. We study the random variable $f_0(T_n)$, the number of vertices of this chain. Our first result gives an explicit expression for the bivariate generating function of the probabilities $\mathbb{P}(f_0(T_n)=k+2)$ in terms of the Gaussian hypergeometric function. 
Building on this analytic representation, we apply a careful singularity analysis to derive a variety of limit theorems for $f_0(T_n)$, including a quantitative central limit theorem, a large deviation principle as well as a precise asymptotics for the probabilities $\mathbb{P}(f_0(T_n)=k+2)$. Conceptually, our results establish a novel bridge between stochastic geometry and methods from analytic combinatorics.

  
  \smallskip\noindent
  \textbf{MSC 2020.} Primary:
  %05A16 %Asymptotic enumeration
  52A22, % Random convex sets and integral geometry (aspects of convex geometry)
  60D05; % Geometric probability and stochastic geometry
  Secondary:
  05A15, %Exact enumeration problems, generating functions
  32A05, %Power series, series of functions of several complex variables
  33C05, % Classical hypergeometric functions
  60F10. % Large deviations
\end{abstract}

%\tableofcontents

\section{Introduction}

\subsection{Motivation}

The most classical model for a random polytope arises by taking the convex hull $K_n$ of $n$ independently and uniformly distributed random points in a convex body $K\subset\RR^d$. The study of these random polytopes is one of the central themes of geometric probability, tracing its origins to questions first posed by Sylvester in the 19th century. The asymptotic study of random polytopes started with the pioneering works \cite{RenyiSulanke1,RenyiSulanke2} of Rényi and Sulanke in the 1960s on random polygons in the plane. Core questions include the asymptotic behavior, as $n\to\infty$, of basic geometric and combinatorial functionals of the random polytope such as the missing volume $\vol(K\setminus K_n)$ or the number of vertices $f_0(K_n)$.

For convex bodies $K$ in $\RR^d$ with a smooth boundary, one now has sharp asymptotics for expectations, variance bounds, and central limit theorems for these basic functionals, while for polytopal $K$ the behavior is different and governed by the combinatorics of the normal fan of the underlying container polytope. For example, regarding the number of vertices, the picture that has emerged is by now classical: for convex bodies with a smooth boundary, the expected number of vertices grows of the order $c_d\Omega(K)n^{d-1\over d+1}$, whereas in the polytopal case it only grows of the logarithmic order $\widetilde{c}_dT(K)(\log n)^{d-1}$. 
Here, $c_d$ and $\widetilde{c}_d$ are explicit dimension dependent constants, whereas $\Omega(K)$ stands for the affine surface area of $K$ and $T(K)$ is the number of maximal chains of cones in the normal fan if $K$ itself is a polytope. Related variance bounds and central limit theorems have also been developed, see \cite{BaranyReitzner:Variance, CalkaYukich:2014, CalkaYukich:2017, FodorPapvari:2023, FodorVigh:2018, GusakovaReitznerThaele, PardonCLT, ReitznerCLT, Reitzner:EfronStein}. We also refer to the survey articles \cite{BaranySurvey,HugSurvey,ReitznerSurvey,SchneiderSurvey, SW:2023} for an overview and a general guide to the literature.

In dimension two, several phenomena admit exact formulas and more explicit descriptions. A classical family of questions asks for the probability that $n$ independent and uniform random points in a fixed convex set are in convex position, i.e., all lie on the boundary of their convex hull. For the following two convex polygons this was solved in \cite{Valtr:1995,Valtr:1996}:
\begin{align*}
\PP(\text{$n$ uniform points in a triangle are in convex position})
&=\frac{2^n}{(2n)!}\, \frac{(3n-3)!}{((n-1)!)^3} = n^{-2n(1+o(1))}, \\
\PP(\text{$n$ uniform points in a parallelogram are in convex position})
&=\bigg(\frac{1}{n!}\,\frac{(2n-2)!}{((n-1)!)^2}\bigg)^{\!\!2} = n^{-2n(1+o(1))}.
\end{align*}
Analogous results and sharp estimates have been obtained for other convex bodies in $\RR^2$. For instance, the disk has been studied in~\cite{Marckert:2017}, while in \cite{BaranySylvester:1999} it was shown for all convex sets $K\subseteq\RR^2$ with unit area that
$$
\lim_{n\to\infty}n^2\,\PP(\text{$n$ uniform points in $K$ are in convex position})^{1/n} = \frac{e^2}{4}A(K)
$$
where $A(K)=\sup\{\Omega(L):L\subseteq K\}$ is the extremal affine perimeter, that is, the supremum of the affine perimeter $\Omega(L)$ of all convex sets $L\subseteq K$; see \cite{EglerWerner:2024} for a new functional version of this notion. 

A key structural insight, going back to at least Groeneboom \cite{Groeneboom1988} and developed further in later works, is that for polygonal convex bodies the main contribution to the boundary complexity of the random convex hull comes from nearly independent corner processes; see Figure \ref{fig:square} for an illustration in the case of a square. This motivates isolating and analyzing a single corner, which naturally leads to the model of \textit{random convex chains} that we discuss next.

\begin{figure}[t]
    \centering
    \includegraphics[width=0.4\linewidth]{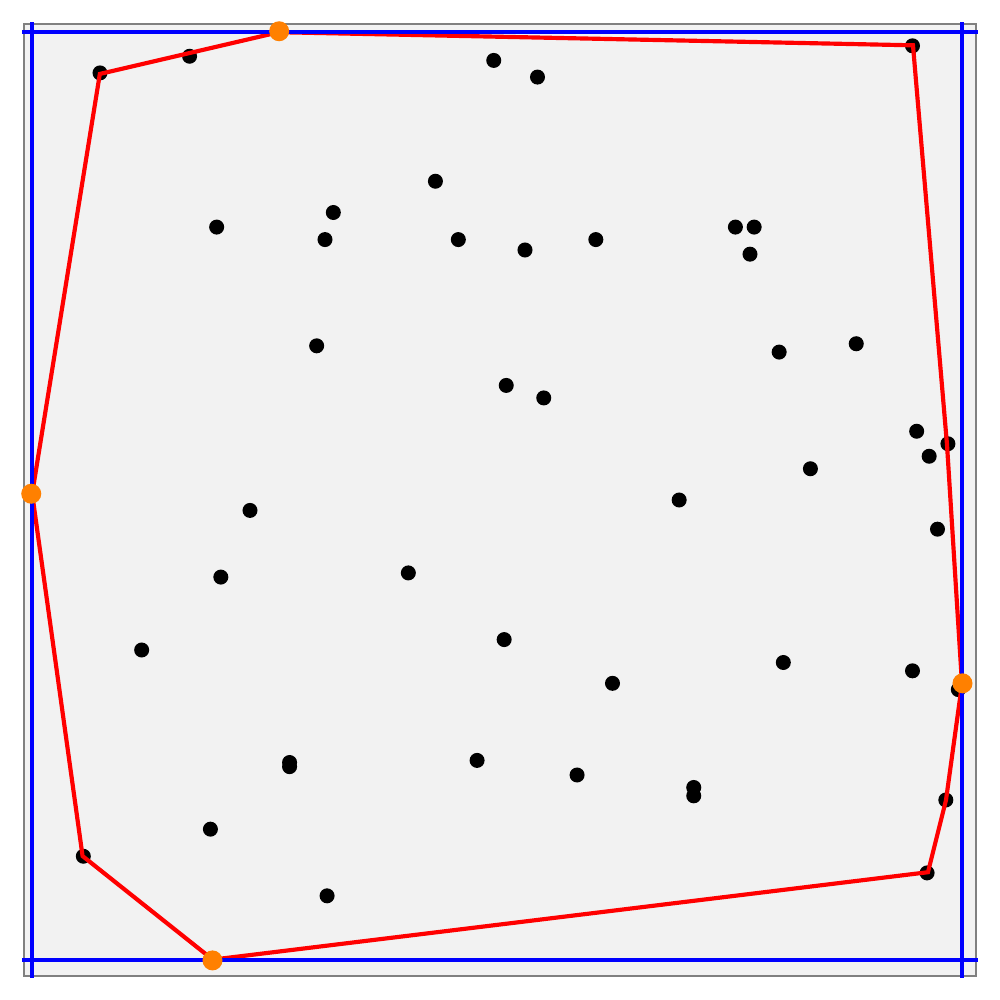}
    \caption{Decomposition of a random polygon in a square into four convex chains.}
    \label{fig:square}
\end{figure}

\medskip

Consider the triangle $T\subset \RR^2$ spanned by the vertices $\bv_0:=(0,0), \bv_1:=(1,0)$ and $\bv_2:=(0,1)$. Let $n\in\NN$, generate independent and uniform random points $\bX_1,\dotsc,\bX_n\in T$, and consider the convex hull
$$
T_n:=\conv \{\bv_1,\bv_2, \bX_1,\dotsc,\bX_n\},
$$
see Figure \ref{fig:chains}. The lower boundary of $T_n$ is known as a random convex chain.
Buchta \cite{Buchta:2006} established that the probability that $T_n$ has exactly $k+2$ vertices (including $\bv_1$ and $\bv_2$) is given by
\begin{equation}\label{eq:pnk}
    p_k^{(n)} = \PP(f_0(T_n)=k+2) = \sum_{i_1+\dotsc+i_k = n}\prod_{m=1}^k \frac{i_m}{\binom{I_m+1}{2}},\qquad I_m:=i_1+\ldots+i_m,
\end{equation}
for $k\in\{0,\dotsc,n\}$, where the sum is taken over $k$-tuples $(i_1,\dotsc,i_k)$ of positive integers. The particular value
\begin{equation}\label{eq:pnn}
    p_n^{(n)} = \prod_{m=1}^n{1\over\binom{m+1}{2}} = \frac{2^n}{n!(n+1)!} = n^{-2n(1+o(1))}
\end{equation}
that the $n$ points are in convex position has previously been computed in \cite{BaranyRoteSteigerZhang:2000}. Buchta also computed the moments and provided recursion schemes for higher moments, together with asymptotics for the mean and variance in \cite{Buchta:2012}. The paper \cite{GT:2021} developed Buchta’s approach further by studying the probability generating functions
\begin{equation*}
G_n(t):=\sum_{k=1}^n p_k^{(n)}\,t^k,\qquad t\in\RR.
\end{equation*}
These authors showed a three-term recursion for the $G_n$, identified a link to families of orthogonal polynomials, and proved the remarkable fact that for each fixed $n$ all zeros of $G_n$ are non-positive and real. As a consequence, the coefficient sequence $\{p_k^{(n)}\}_{k=1}^n$ is a so-called Pólya frequency (PF) sequence, which entails strong probabilistic corollaries as described in the survey article \cite{Pitman}.

\begin{figure}[t]
    \centering
    \includegraphics[width=0.3\linewidth]{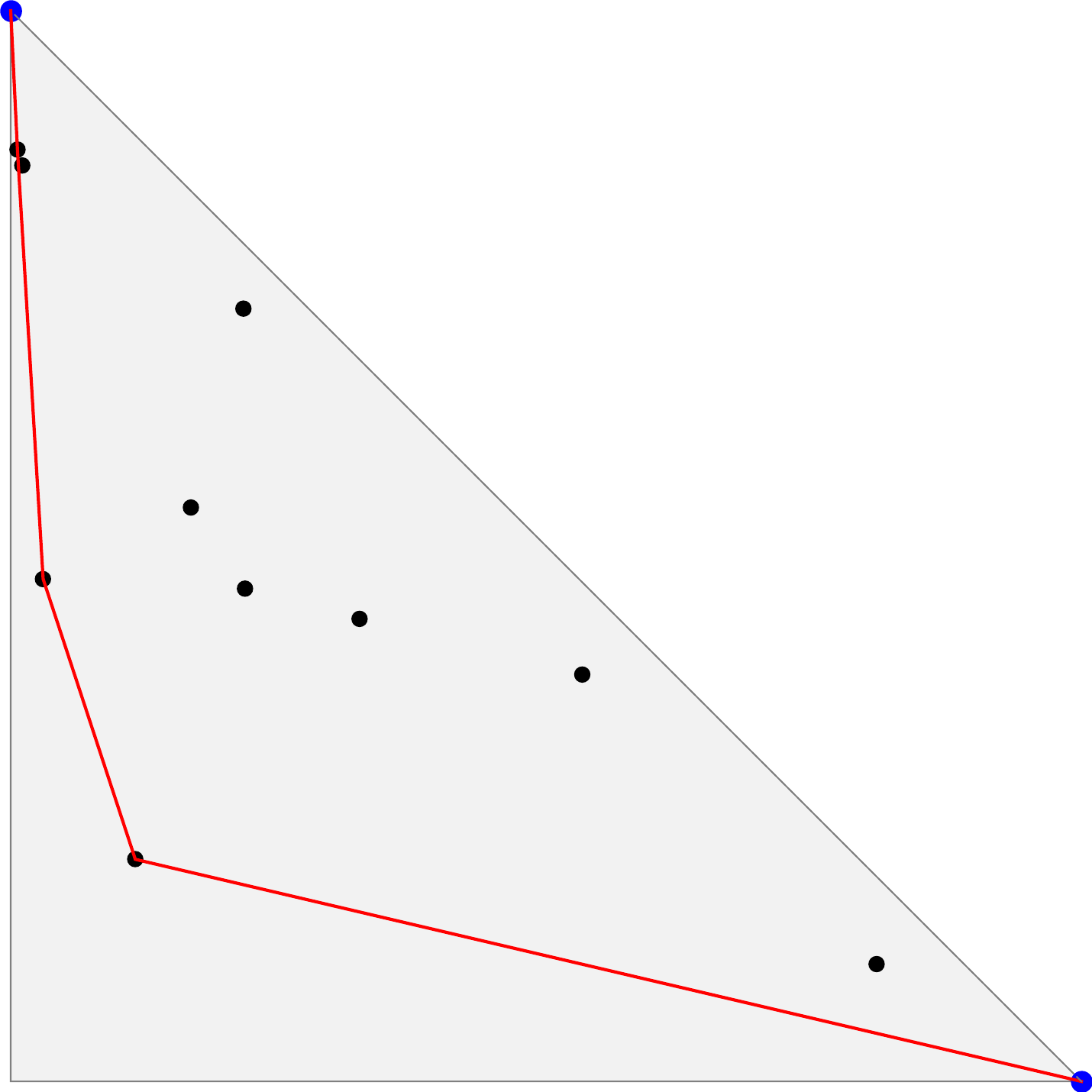}\quad
    \includegraphics[width=0.3\linewidth]{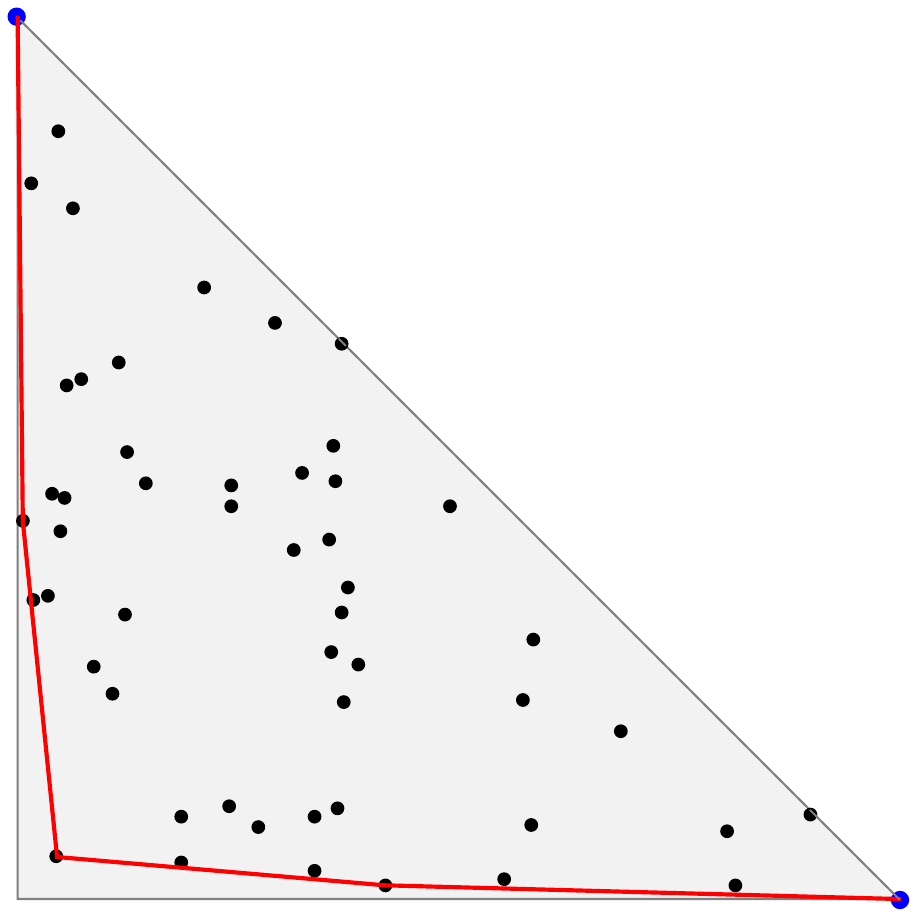}\quad
    \includegraphics[width=0.3\linewidth]{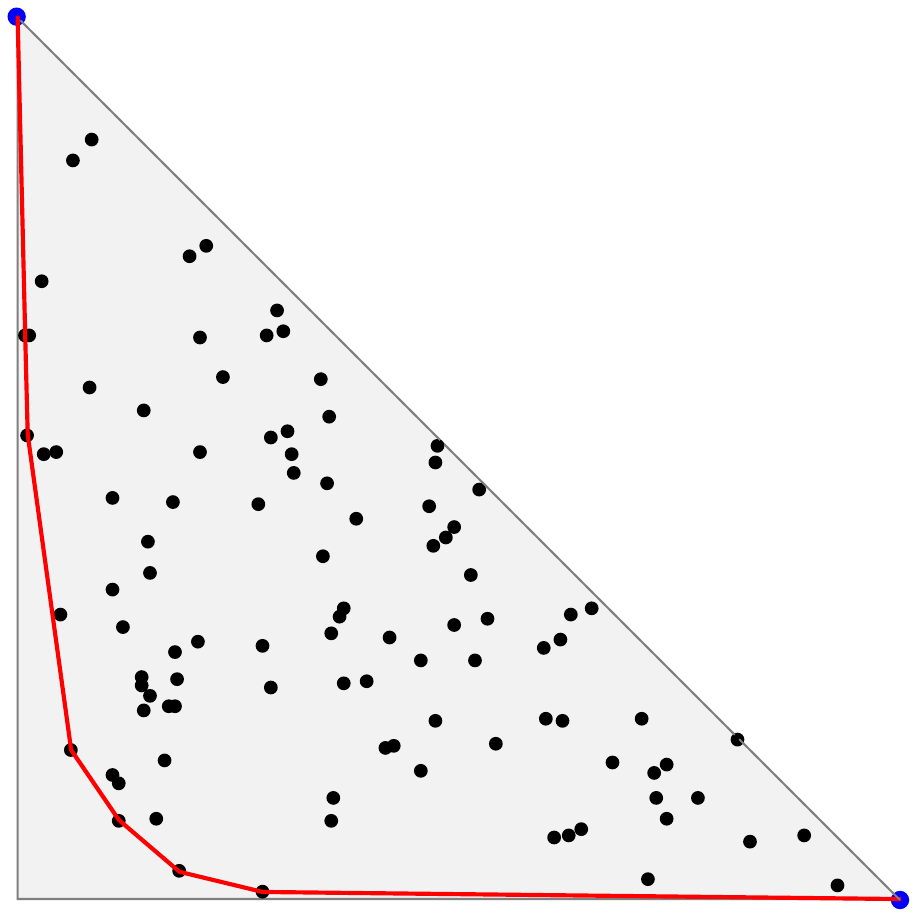}
    \caption{Simulations of a random convex chain generated by $n=10$ (left), $n=50$ (middle) and $n=100$ (right) points.}
    \label{fig:chains}
\end{figure}

In the present paper, we go one step further and revisit the random convex chain in a triangle through the lens of analytic combinatorics, complementing Buchta’s exact distribution formula \cite{Buchta:2006} and the approach of \cite{GT:2021}.

\subsection{Overview of main results}

Our first main result is an explicit and closed form for the \emph{bivariate} generating function
\begin{equation*}
    F(u,z) := \sum_{n=0}^\infty G_n(z) u^n = 1+ uz + \sum_{n=2}^\infty \sum_{k=1}^n p_k^{(n)} z^k u^n,\qquad u,z\in\CC,
\end{equation*}
expressed in terms of the Gaussian hypergeometric function ${_2F_1}$, see Theorem \ref{thm:SolODE}. This analytic representation packages the whole family of distributions of the random variables $\{f_0(T_n)\}_{n\ge 0}$ into a single object. Building on this, we carry out a detailed singularity analysis of $F(u,z)$ to obtain a collection of probabilistic limit theorems for $f_0(T_n)$:
\begin{itemize}
    \item[(i)] Precise cumulant asymptotics for the number of vertices, complementing the simple moment asymptotics of \cite{Buchta:2012}. Namely, for every integer $r\ge1$, we show that the $r$-th cumulant $\kappa_n^{(r)}$ of $f_0(T_n)$ satisfies
    \[
    \kappa_n^{(r)}
    \sim \Bigg(\sum_{k=1}^{r} (-1)^{k-1}\,S(r,k)\,
    \frac{2^{k}\,(2k-2)!}{3^{2k-1}\,(k-1)!}\Bigg)\, \log n,\qquad \text{as }n\to\infty,
    \]
    where $S(r,k)$ are the Stirling numbers of the second kind, see Theorem \ref{thm:CumulantAsymptotics}.
    
    \item[(ii)] A quantitative central limit theorem for the number of vertices with an explicit error term. More precisely, for 
    \[
        Z_n:=\frac{f_0(T_n)-\EE f_0(T_n)}{\sqrt{\var(f_0(T_n))}},
    \]
    we obtain
    \[
    \sup_{x\in\RR}\bigl|\PP(Z_n\le x)-\Phi(x)\bigr|=O\Big(\tfrac{1}{\sqrt{\log n}}\Big),
    \]
    where $\Phi$ is the standard Gaussian distribution function, see Theorem \ref{thm:CLT-rate} and Figure \ref{fig:Histogram}.
    
    \item[(iii)] A moderate and large deviation principle derived from the dominant singularity, identifying the exponential decay of off-typical fluctuations, are presented in Theorem \ref{thm:MDP} and Theorem \ref{thm:LDP}. More precisely, we show that the sequence of random variables $\tfrac{f_0(T_n)}{\log n}$ satisfies a large deviation principle on $\RR$ with speed $\log n$
    and a good rate function
    \[
    I(x)=
    \begin{cases}
    +\infty, & x<0,\\[2mm]
    x\log\!\big(\tfrac{x}{2}\big)+x\,\operatorname{arsinh}(2x)-x-\tfrac12\sqrt{1+4x^2}+\tfrac32, & x\ge0,
    \end{cases}
    \]
    see Figure \ref{fig:LDP} for a plot of the strictly convex function $I(x)$.
    
    \item[(iv)] Precise asymptotics for the probabilities $p_k^{(n)}=\PP(f_0(T_n)=k+2)$. For example, for each fixed integer $k\in\ZZ_{\geq 0}$, in Theorem \ref{thm:probab} we show that
    \[
    p_{k+1}^{(n)}
    \sim \frac{2^{k+1}}{k!}\,\frac{(\log n)^{k}}{n}\qquad \text{as }n\to\infty.
    \]
\end{itemize}

\begin{figure}[t]
    \centering
    \includegraphics[width=0.45\linewidth]{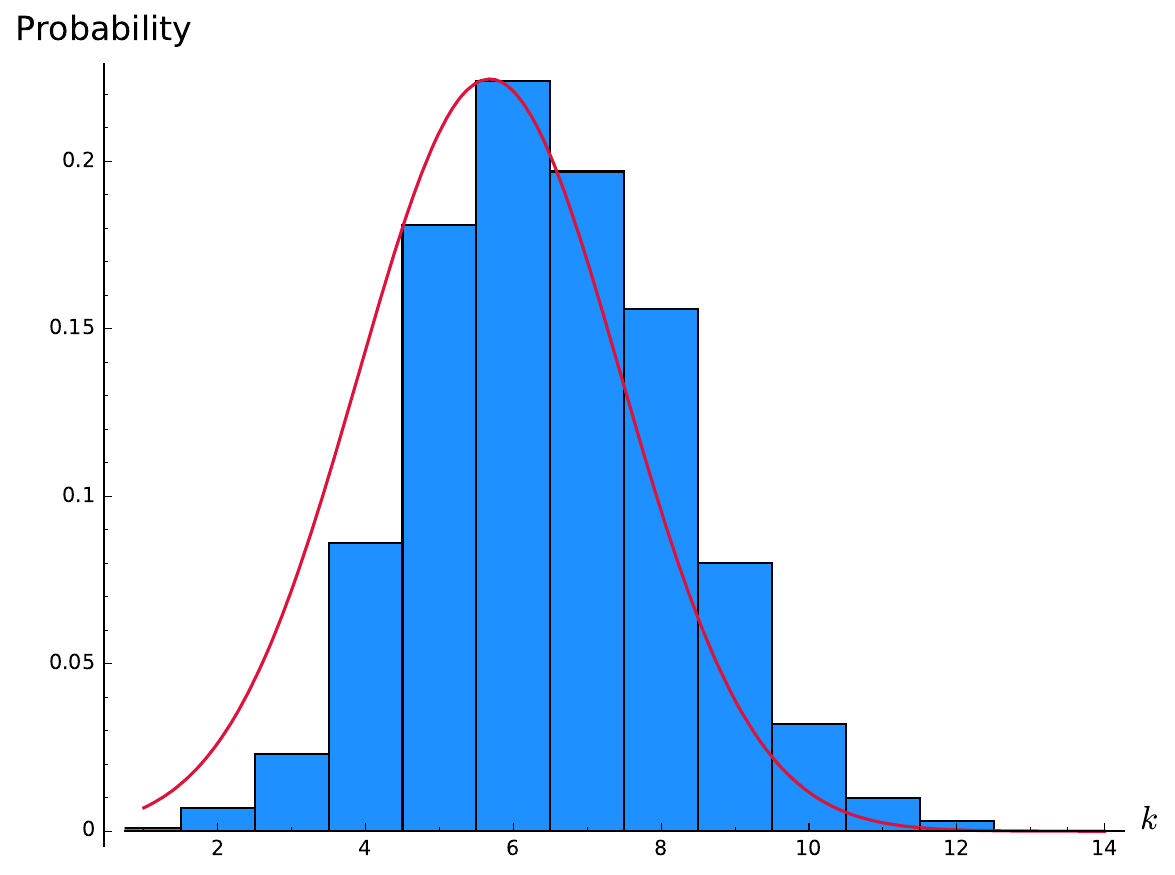}\quad
    \includegraphics[width=0.45\linewidth]{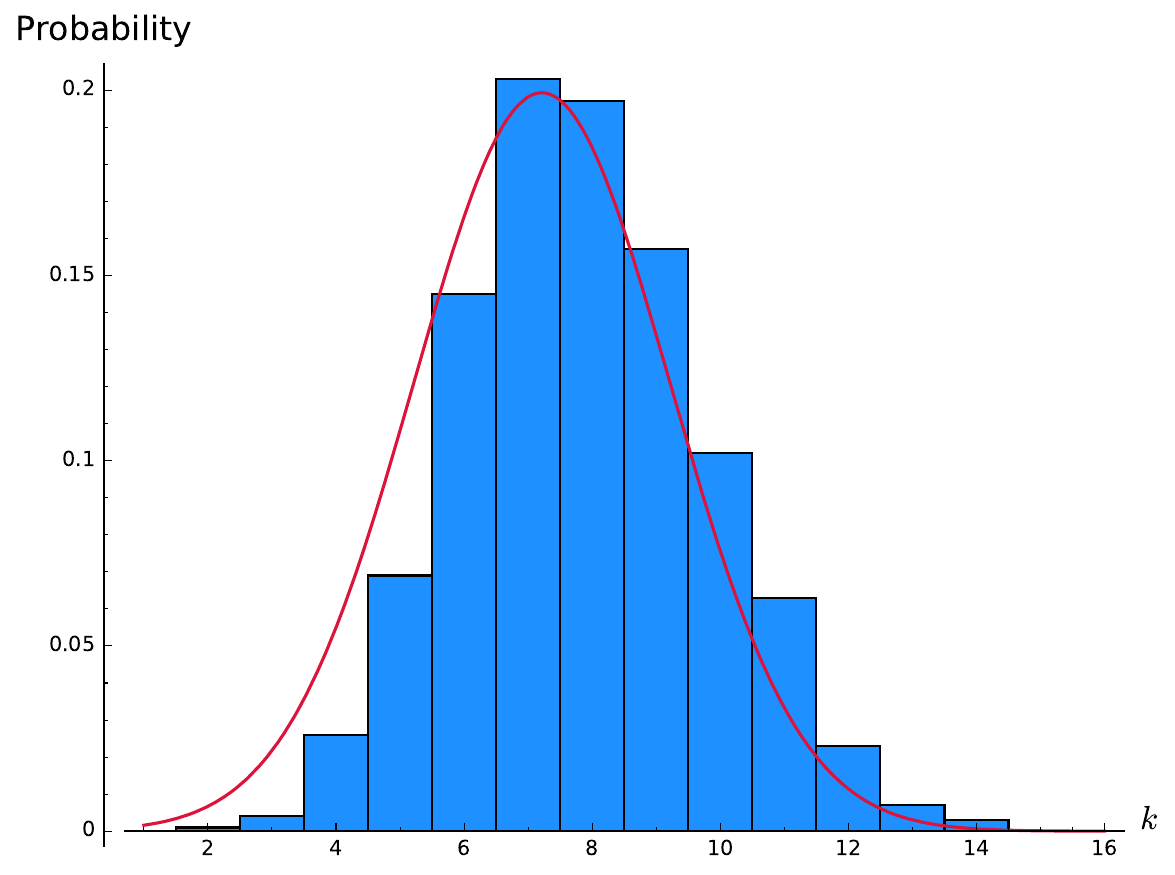}%\quad
    \caption{Histogram of the probability $p_k^{(n)}=\PP(f_0(T_n)=k+2)$ generated by $1000$ runs with $n=5,000$ (left) and $n=50,000$ (right). The red curve is the plot of the Gaussian density function $\varphi_n(x) = \frac{1}{\sigma_n\sqrt{2\pi}} \exp\left(-\frac{(x-\mu_n)^2}{2\sigma_n^2}\right)$, where $\mu_n = \frac{2}{3}\log n \sim \EE f_0(T_n)$ and $\sigma_n = \sqrt{\frac{10}{27} \log n} \sim \sqrt{\var f_0(T_n)}$, see Section~3.} 
    \label{fig:Histogram}
\end{figure}

In particular, our method yields simple and direct derivations of several consequences previously obtained via PF–factorizations and cumulant bounds in \cite{GT:2021}, while at the same time providing finer information that was not within the reach of previously existing methods. A key highlight in our results is the new large deviation principle we establish for the sequence of random variables $f_0(T_n)/\log n$, whose rate function admits a simple probabilistic interpretation. 
To the best of our knowledge, this work presents the first large deviation principle specifically for random convex hulls of random points.
The second highlight is the precise asymptotics for the probabilities $p_k^{(n)}$, which we obtain not only for fixed $k$, but also for sequences $k_n$ that tend to infinity with $n$.

On a conceptual level, our results establish a novel link between stochastic geometry and analytic combinatorics. We demonstrate how distributional questions about random polygons can be reformulated as analytic problems about generating functions, with their singularities encoding the relevant limit theorems. This viewpoint not only unifies and simplifies earlier approaches, but also extends the reach of analytic combinatorics to a setting that lies outside its typically discrete domain. Whereas analytic combinatorics is traditionally concerned with finite combinatorial classes and exact and asymptotic enumerations, the present setting is genuinely continuous. The fact that analytic combinatorics methods nevertheless yield precise probabilistic information highlights both the flexibility of the analytic framework and its potential to address a broader class of problems in geometric probability.

\medskip

The remaining parts of this paper are structured as follows. In Section \ref{sec:BivariageGF} we show that the bivariate generating function $F(u,z)$ satisfies an ordinary differential equation (ODE) of second order, whose solution will be expressed in terms of the Gaussian hypergeometric function. Using the method of transfer of singularities, this leads to detailed asymptotic expansions of $F(u,z)$. These are used as a common basis for all probabilistic results that we develop in Section \ref{sec:Applications}.

\medskip

\textbf{Note.} Shortly after completing this paper, we became aware that independently Gusakova and Muranova \cite{GusMur} also obtained an explicit expression for the bivariate generating function of the probabilities $p_k^{(n)}=\PP(f_0(T_n)=k+2)$, in order to investigate the moments of the area cut off from the triangle $T$ by the random chain $T_n$. However, apart from Theorem~\ref{thm:SolODE}, our results appear to be disjoint from those presented in \cite{GusMur}.

\section{Bivariate generating function and singularity transfer}\label{sec:BivariageGF}

\subsection{Set-up and ODE}

We recall that $T$ is the triangle with vertices $\bv_0:=(0,0), \bv_1:=(1,0)$ and $\bv_2:=(0,1)$ and that $T_n$ denotes the convex hull of $\bv_1,\bv_2$ and $n\geq 0$ independent and uniformly distributed random points in $T$. The lower boundary of $T_n$ is the random convex chain we consider.

In the proof of Theorem 1 in \cite{Buchta:2012} it was shown that the probabilities $p_k^{(n)}=\PP(f_0(T_n)=k+2)$ in \eqref{eq:pnk} satisfy the recurrence relation
\begin{equation}\label{eq:RectProbab}
    \binom{n+1}{2} p_k^{(n)} - 2 \binom{n}{2} p_k^{(n-1)}+\binom{n-1}{2} p_k^{(n-2)} = p_{k-1}^{(n-1)},
\end{equation}
see also \cite[Equation (3)]{GT:2021}. In Theorem 1 of \cite{GT:2021} this was used to derive a recurrence relation for the probability generating polynomials of the random variables $f_0(T_n)$. More precisely, the probability generating polynomial of $f_0(T_n)$ is defined as
\begin{equation*}
    G_n(z) := \sum_{k=1}^n p_k^{(n)} z^k \qquad \text{for $z\in\CC$ and $n\geq 2$}.
\end{equation*}  
%Note that
%\begin{equation*}
%    G_n(1) = \sum_{k=1}^n p_k^{(n)} = 1, \qquad \text{for all $n\geq 2$}.
%\end{equation*}
For convenience we extend this family by setting $G_1(z):=z$ and $G_0(z):=1$. In the language of generating polynomials, the recurrence relation \eqref{eq:RectProbab} translates into
\begin{equation}\label{eq:RectGF}
    G_n(z) = (a_n z+ b_n) G_{n-1}(z) - c_n G_{n-2}(z), \qquad \text{$n\geq 2$, $z\in\CC$},
\end{equation}
with
\begin{equation*}
    a_n := \frac{2}{n(n+1)}, \quad b_n := 2\frac{n-1}{n+1}\quad \text{and} \quad c_n := \frac{(n-1)(n-2)}{n(n+1)},
\end{equation*}
see \cite[Theorem 1]{GT:2021}. This recurrence was used in \cite{GT:2021} for \textit{real} $z$ to show that each $G_n$ is a polynomial with only real roots.

In the present paper we go one step further and consider the bivariate generating function of the random variables $f_0(T_n)$ defined as
\begin{equation*}
    F(u,z) := \sum_{n=0}^\infty G_n(z) u^n = 1+ uz + \sum_{n=2}^\infty \sum_{k=1}^n p_k^{(n)} z^k u^n.
\end{equation*}
We emphasize that we regard $F(u,z)$ as an element of the ring $\CC[[u,z]]$ of formal bivariate power series over $\CC$. On the other hand, we note that for $z=1$,
\begin{equation*}
    F(u,1) = \sum_{n=0}^\infty u^n = \frac{1}{1-u}.
\end{equation*}
Thus $F(u,1)$ converges for all $u\in\CC$ with $|u|<1$ and extends naturally to the meromorphic function $(1-u)^{-1}$ on $\CC\setminus\{1\}$ with a pole of first order at $u=1$. Similarly, for $|u|<1$ and $|z|\leq 1$ we find that
\begin{equation*}
    |F(u,z)| \leq \frac{1}{1-|u|},
\end{equation*}
and therefore $F(u,z)$ also defines a bivariate jointly analytic function at least in this domain. The precise region in $\CC^2$ where $F(u,z)$  defines a jointly analytic function will be identified later We refer to \cite[Ch.\ III.~1.]{FS:2009} for background on bivariate generating functions.

Our first result shows that the bivariate generating function $F$ satisfies a linear ODE of second order. To formulate it, we use the differential operator $\Xi := u \frac{\partial }{\partial u}$.

\begin{theorem}[ODE for the bivariate generating function]\label{thm:ODEtheta}
The bivariate generating function $F$ satisfies 
    \begin{equation*}
    (1-u)^2 (\Xi+1)\Xi F = 2uz F.
\end{equation*}
\end{theorem}
\begin{proof}
We start by rewriting the recurrence relation \eqref{eq:RectGF} in the form
\begin{equation}\label{eqn:recurrence}
    \frac{n(n+1)}{2} G_n(z) - n(n-1) G_{n-1}(z)  + \frac{(n-1)(n-2)}{2} G_{n-2}(z) = z G_{n-1}(z),
\end{equation}
for all $n\geq 2$ and $z\in\CC$. Next, we compute
\begin{equation*}
    \Xi F = \sum_{n=1}^\infty G_n(z) n u^n, \quad
    (\Xi+1)\Xi F = \sum_{n=1}^\infty G_n(z) n(n+1) u^n = 2uz + \sum_{n=2}^\infty G_n(z) n(n+1) u^n,
\end{equation*}
and thus
\begin{equation*}
    u(\Xi+1)\Xi F = \sum_{n=2}^\infty G_{n-1}(z) n(n-1) u^n,\quad 
    u^2 (\Xi+1)\Xi F = \sum_{n=3}^\infty G_{n-2}(z) (n-1)(n-2) u^n.
\end{equation*}
Using \eqref{eqn:recurrence} this leads to
\begin{align*}
    &\frac{1}{2} (1-u)^2 (\Xi+1)\Xi F - uz \\
    & = \sum_{n=2}^\infty \left[\frac{n(n+1)}{2} G_n(z) - n(n-1) G_{n-1}(z) + \frac{(n-1)(n-2)}{2} G_{n-2}(z)\right] u^n\\
    & = \sum_{n=2}^\infty z G_{n-1}(z) u^n = uz F- uz.
\end{align*}
As a result, we have
\begin{equation*}
    (1-u)^2 (\Xi+1)\Xi F = 2uz F,
\end{equation*} 
which completes the proof.
\end{proof}

It will be convenient to rephrase the result of Theorem \ref{thm:ODEtheta} in a traditional way. For a function $F\in\CC[[u,z]]$ we write $F_u$ for ${\partial\over\partial u}F(u,z)$.

\begin{corollary}[ODE for the bivariate generating function]
The function $F$ satisfies the homogeneous linear ODE of second order
\begin{equation}\label{eqn:ODE}
    F_{uu} + P F_u  + Q F = 0
\end{equation}
with initial conditions $F(0,z)=1$, $F_u(0,z)=z$, where the coefficients are the meromorphic functions
\begin{equation*}
    P := \frac{2}{u}, \qquad Q:= -\frac{2z}{u(u-1)^2}.
\end{equation*}
The ODE has regular singularities at $u=0$, $u=1$ and $u=\infty$.
\end{corollary}
\begin{proof}
We have $\Xi F=uF_u$ and
$$
(\Xi+1)\Xi F=uF_u+u(uF_u)_u = uF_u+u(F_u+uF_{uu}) = 2uF_u+u^2F_{uu}.
$$
Thus
$$
(1-u)^2(\Xi+1)\Xi F-2uzF = (1-u)^2(2uF_u+u^2F_{uu})-2uzF = 0.
$$
Division by $u^2(1-u)^2$, which is legitimate for $u\notin\{0,1\}$, yields the result as a formal identity in $\CC[[u,z]]$.

The ODE has a regular singularities at $u=0$ since it is a pole of first order of $P$ and $Q$ and at $u=1$ since it is a pole of second order of $Q$. Substituting $v=1/u$ gives the ODE
\begin{equation*}
    G_{vv} - \frac{2z}{v(v-1)^2} G = 0,
\end{equation*}
for $G(v) = F(1/v,\cdot)$, which has a regular singularity at $v=0$. Thus $u=\infty$ is also a regular singularity of \eqref{eqn:ODE}.
\end{proof}

\subsection{Explicit form of the bivariate generating function}

Our goal in this section is to derive an explicit solution of the ODE \eqref{eqn:ODE}. For that purpose, we need the Gaussian hypergeometric function, which is defined as the power series
\begin{equation}\label{def:hypergeo}
    {_2F_1}(a,b;c;z) := \sum_{n=0}^\infty{a^{\overline{n}} b^{\overline{n}} \over c^{\overline{n}}}{z^n\over n!} = 1 +{ab\over c}z+{ab(a+1)(b+1)\over c(c+1)}{z^2\over 2}+ \ldots,
\end{equation}
where $q^{\overline{n}}:=q(q+1)\cdots(q+n-1)$ is the rising factorial with the convention $q^{\overline{0}}:=1$. 
Here, $a,b,c\in\CC$ with $c\notin\{0,-1,-2,\ldots\}$. The series converges absolutely in the unit disc $|z|<1$. By analytic continuation, ${_2F_1}(a,b;c;z)$ extends uniquely to a multivalued analytic function on $\CC\setminus[1,\infty)$, see \cite[Section 5.2]{Temme}.

\begin{theorem}[Solution of the ODE for the bivariate generating function]\label{thm:SolODE}
The unique solution of the ODE \eqref{eqn:ODE} satisfying the given initial conditions is 
\begin{equation*}
    F(u,z) = \frac{{_2F_1}(1-\alpha,-\alpha;2;u)}{(1-u)^{\alpha}}, \qquad \alpha:=\frac{\sqrt{8z+1}-1}{2}.
\end{equation*}
It extends to a jointly analytic function on $\{(u,z)\in\CC^2:u\notin [1,\infty),z\notin(-\infty,-1/8]\}$.
\end{theorem}
\begin{proof}
The ODE \eqref{eqn:ODE} is a Riemann--Papperitz equation and can be transformed into a hypergeometric ODE as explained in \cite[p.\ 116]{Temme}.
To do so, we first write the ODE in the form
\begin{equation}\label{eqn:ODE2}
    u(1-u)^2 F_{uu} + 2(1-u)^2 F_{u} - 2z F  =0, \qquad F(0,z)=1, F_u(0,z)=z.
\end{equation} 
Solutions of this ODE can therefore have singularities at $u=0$, $u=1$ or at $u=\infty$. The solution we are interested in is regular at $u=0$ and converges on the unit disk $|u|<1$. Thus, the asymptotic behavior of the solution $F$ around $u=1$ must be of the form $C/(1-u)^\alpha$, with suitable $C$ and $\alpha$. We therefore make the ansatz
$$
F(u,z) = {A(u,z)\over (1-u)^{\alpha(z)}}.
$$
Then
\begin{align*}
    F_u &= A_u (1-u)^{-\alpha} +\alpha A (1-u)^{-\alpha-1},\\
    F_{uu} &= A_{uu} (1-u)^{-\alpha} + 2\alpha A_u (1-u)^{-\alpha-1} + \alpha(\alpha+1) A (1-u)^{-\alpha-2}.
\end{align*}
Plugging this into \eqref{eqn:ODE2} we derive
\begin{align*}
    0 &= u(1-u)^2\big(A_{uu}(1-u)^{-\alpha}+2\alpha A_u (1-u)^{-\alpha-1}+\alpha(\alpha+1)A(1-u)^{-\alpha-2}\big) \\
    &\qquad + 2(1-u)^2\big(A_u(1-u)^{-\alpha}+\alpha A (1-u)^{-\alpha-1}\big) -2zA(1-u)^{-\alpha}\\
    &=(1-u)^{-\alpha}\Big[u(1-u)^2A_{uu}+(2\alpha u(1-u)+2(1-u)^2)A_u + (\alpha(\alpha+1)u+2\alpha (1-u)-2z)A\Big]
\end{align*}
Multiplying with $(1-u)^{\alpha}$, we conclude that $A$ satisfies the ODE 
\begin{equation*}
    u(1-u)^2 A_{uu} + (2\alpha u+ 2(1-u)) (1-u) A_u + (\alpha (\alpha+1) u + 2\alpha (1-u) - 2z) A = 0.
\end{equation*}
Taking $u\to 1^-$ we derive the characteristic (indicial) equation 
\begin{equation}\label{eqn:indicial}
    \alpha(\alpha+1) = 2z.
\end{equation}
Using this, we may factor another $(1-u)$ and arrive at
\begin{align*}
     (1-u)\Big[u(1-u)A_{uu} + (2\alpha u + 2(1-u))A_u + 2\alpha A\Big] = 0,
\end{align*}
or equivalently,
\begin{equation*}
    u(1-u) A_{uu} + (2- (1-2\alpha) u) A_u - \alpha(\alpha-1)  A = 0, \qquad A(0,z)=1, A_u(0,z) = z-\alpha.
\end{equation*}

This is a particular instance of the hypergeometric differential equation
$$
u(1-u)H_{uu} + [c-(a+b+1)u]H_u - ab H = 0
$$
with parameters
\begin{equation*}
    a = 1-\alpha, \quad 
    b = -\alpha, \quad \text{ and }\quad 
    c = 2.
\end{equation*}
A solution is given by the Gaussian hypergeometric function
\begin{equation*}
    A(u,z) = {_2F_1}(1-\alpha,-\alpha;2;u)
\end{equation*}
for the root $\alpha=-\tfrac{1}{2}+\tfrac{1}{2}\sqrt{8z+1}$ of \eqref{eqn:indicial}, see \cite[Section 5.3]{Temme}. We may verify that the initial conditions hold by calculating
\begin{equation*}
    A(0,z) = {_2F_1}(1-\alpha,-\alpha;2;0) = 1,
\end{equation*}
and
\begin{equation*}
    A_u(0,z) = \frac{\alpha(\alpha-1)}{2} {_2F_1}(2-\alpha,1-\alpha;3;0) = \frac{\alpha(\alpha-1)}{2} = z-\alpha,
\end{equation*}
where the last equality follows from the characteristic equation \eqref{eqn:indicial}. 

We observe that the choice of the other root of the characteristic equation \eqref{eqn:indicial} is not compatible with the initial conditions
Indeed, if we pick $\alpha=-\tfrac{1}{2}-\tfrac{1}{2}\sqrt{8z+1}$, then the prefactor $(1-u)^{-\alpha}$ grows like $(1-u)^{(1+\sqrt{1+8z})/2}$ near $u=1$, which forces $A(0,z)=0$ and contradicts $F(0,z)=1$. Thus, the unique solution of the ODE \eqref{eqn:ODE2} is
\begin{equation*}
    F(u,z) = \frac{{_2F_1}(1-\alpha,-\alpha;2;u)}{(1-u)^{\alpha}}
\end{equation*}
for $\alpha=-\tfrac{1}{2}+\tfrac{1}{2}\sqrt{8z+1}$. 

The natural brach cut for ${_2F_1}$ is $[1,\infty)$, whereas the square root $\sqrt{8z+1}$ introduces a  branch cut starting from $-1/8$ and extending to $-\infty$. As a result, the solution extends uniquely to a multivalued jointly analytic function on $\{(u,z)\in\CC^2:u\notin [1,\infty),z\notin(-\infty,-1/8]\}$.
\end{proof}

\begin{remark}[Alternative representations of the solution]
    Using the functional relations of the hypergeometric function in \cite[Sec.\ 5.2.1]{Temme} we can rewrite our solution as follows:

    \begin{enumerate}
        \item By Euler's transformation formula for the hypergeometric function, we have
    \begin{equation*}
        F(u,z) 
        %= \frac{{_2F_1}(-\alpha,1-\alpha;2;u)}{(1-u)^\alpha} 
        = (1-u)^{1+\alpha} {_2F_1}(2+\alpha,1+\alpha;2;u).
    \end{equation*}

    \item Using the Pfaff transformation, we may also write the solution as
    \begin{equation*}
        F(u,z) = {_2F_1}\left(-\alpha,1+\alpha;2;\frac{u}{1-u}\right).
    \end{equation*}
     
    \end{enumerate}
\end{remark}

\subsection{Asymptotic estimates by transfer of singularities}

In this section, we derive a precise version of the asymptotic behavior of the probability generating function $G_n(z)$ of the random variable $f_0(T_n)$ as $n\to\infty$. These asymptotics are the key tool for the probabilistic results presented in the next section. To this end, we first establish for a given $z$ (or equivalently $\alpha$) an asymptotic formula for a bivariate generating function $F(u,z)$ as $u\to 1^-$, from which the desired asymptotics for $G_n(z)$ will follow.

\begin{theorem}[Asymptotics for the bivariate generating function]\label{thm:SingExpFuz}
For every $k\in\ZZ_{\ge 0}$ and $I_k:=(k-\tfrac{1}{2},k+\tfrac{1}{2})$, we have for $\alpha=\alpha(z)$,
\begin{equation*}
F(u,z)=
\begin{cases}
\displaystyle
    \frac{K(\alpha)}{(1-u)^{\alpha}}
    +\sum_{j=1}^{k+1} \frac{K(\alpha)a_j(\alpha)}{(1-u)^{\alpha-j}}
    +L(\alpha)(1-u)^{1+\alpha}
    +O\bigl((1-u)^{k+2-\alpha}\bigr), & \text{$\alpha\in I_k$},\\[10pt]
    \displaystyle
    \frac{K(\alpha_0)}{(1-u)^{\alpha_0}}
    +(1-u)^{1+\alpha_0}\big[\,2C_k(\alpha_0)\log(1-u)+O(1)\big]
    +O\bigl((1-u)^{1-\alpha_0}\bigr), & \alpha=\alpha_0=\tfrac12+k.
\end{cases}
\end{equation*}
as $u\to 1^-$, where $C_k(\alpha_0)\in (0,\frac{1}{10})$ is defined by
\begin{equation}\label{eq:ConstCk}
C_k(\alpha_0):=\frac{\Gamma(\alpha_0)^2}{4\pi^2\alpha_0(1+2\alpha_0)\,\Gamma(2\alpha_0)}
=\frac{\Gamma\!\bigl(k+\tfrac12\bigr)^2}{8\pi^2\,(k+1/2)(k+1)\,\Gamma(2k+1)},
\end{equation}
and $K(\alpha),L(\alpha)$ and $a_j(\alpha)$ are defined below in \eqref{def:KandL} and \eqref{def:an}.
\end{theorem}
\begin{proof}
Using a fractional linear transformation
\cite[Eq.~(5.10)]{Temme} for the hypergeometric function, we rewrite the hypergeometric part of the bivariate generating function we derived in Theorem \ref{thm:SolODE} as
\begin{align}\label{eq:SolAfterTrafo}
    \notag{_2F_1}(1-\alpha,-\alpha;2;u) 
    &= K(\alpha) {_2F_1}(1-\alpha,-\alpha;-2\alpha;1-u) \\ 
    &\qquad +
        L(\alpha) (1-u)^{1+2\alpha} {_2F_1}(2+\alpha,1+\alpha;2(1+\alpha);1-u),
\end{align}
where
\begin{equation}\label{def:KandL}
    K(\alpha) := \frac{\Gamma(1+2\alpha)}{\Gamma(1+\alpha)\Gamma(2+\alpha)}\qquad \text{and}\qquad L(\alpha) := \frac{\Gamma(-1-2\alpha)}{\Gamma(-\alpha)\Gamma(1-\alpha)}.
\end{equation}
The function $K(\alpha)$ is well-defined and analytic for $\alpha>-1/2$. To see where $L(\alpha)$ is well-defined, we use the reflection formula $\Gamma(z)\Gamma(1-z)=\pi/\sin(\pi z)$ for $z\notin\mathbb{Z}$. Applying this to each of the gamma functions in $L(\alpha)$ we see that
$$
L(\alpha) = - {\Gamma(\alpha)^2\,\sin^2(\pi \alpha)\over 2\pi\,(1+2\alpha)\Gamma(2\alpha)\,\sin(2\pi\alpha)}.
$$
Next, using $\sin(2z)=2\sin z\cos z$, we arrive at
\begin{equation}\label{eq:Lalternative}
L(\alpha) = -{1\over 4(1+2\alpha)\pi}{\Gamma(\alpha)^2\over\Gamma(2\alpha)}\,\tan(\pi\alpha),
\end{equation}
which shows that $L(\alpha)$ defines an analytic function on $\RR\setminus(\tfrac{1}{2}+\ZZ)$. 

Fix an integer $k\in\ZZ_{\geq 0}$ and recall $I_k=(k-\tfrac{1}{2},k+\tfrac{1}{2})$. For $\alpha\in I_k$, we expand the first hypergeometric function on the right-hand side of \eqref{eq:SolAfterTrafo} at $1-u=0$ up to order $r:=k+1$ and the second hypergeometric function only up to first order. Then, as $u\to1^-$,
\begin{equation*}
    {_2F_1}(1-\alpha,-\alpha;2;u)
    = K(\alpha)\sum_{j=0}^{r} a_j(\alpha)\,(1-u)^j \;+\; L(\alpha)\,(1-u)^{1+2\alpha}\bigl(1+O(1-u)\bigr)
    \;+\; O\bigl((1-u)^{r+1}\bigr),
\end{equation*}
with $a_0(\alpha)=1$, $a_1(\alpha)=\frac{1-\alpha}{2}$, and $a_j(\alpha)$ analytic in $\alpha$ on compact subsets of $I_k$. Multiplying by $(1-u)^{-\alpha}$ gives
\begin{equation*}
    F(u,z)
    = K(\alpha)(1-u)^{-\alpha}
      + \sum_{j=1}^{r} K(\alpha)a_j(\alpha)\,(1-u)^{j-\alpha}
      + L(\alpha)(1-u)^{1+\alpha}\bigl(1+O(1-u)\bigr)
      + O\bigl((1-u)^{r+1-\alpha}\bigr).
\end{equation*}
Since $r+1-\alpha\ge k+2-(k+\tfrac12)=\tfrac12$, the remaining term decays uniformly on compact subsets of $I_k$. This proves the theorem in the case $\alpha\notin\tfrac{1}{2}+\ZZ$.

To derive the asymptotics for half-integers $\alpha\in\tfrac{1}{2}+\ZZ$, we fix $\alpha_0:=\tfrac{1}{2}+k$ for some $k\in\ZZ_{\geq 0}$, put $m:=2(k+1)=1+2\alpha_0$ and expand the first hypergeometric term ${}_2F_1(1-\alpha,-\alpha;-2\alpha;1-u)$ in \eqref{eq:SolAfterTrafo} in a series at $u=1$:
\begin{equation}\label{def:an}
  {}_2F_1(1-\alpha,-\alpha;-2\alpha;1-u)
  \;=\; \sum_{n=0}^\infty a_n(\alpha)\,(1-u)^n,
  \qquad
  a_n(\alpha)\;=\;\frac{(1-\alpha)^{\overline{n}}\,(-\alpha)^{\overline{n}}}{(-2\alpha)^{\overline{n}}\,n!}.
\end{equation}
For $n\le m-1$ the coefficients $a_n(\alpha)$ are analytic at $\alpha_0$, whereas for $n=m$ we have
\begin{equation*}
(-2\alpha)^{\overline{m}} \;=\; \frac{\Gamma(-2\alpha+m)}{\Gamma(-2\alpha)},
\end{equation*}
which vanishes at $\alpha_0$ because $-2\alpha_0=-(2k+1)$ is an integer. Hence $a_m(\alpha)$ has a simple pole at $\alpha_0$. This yields
\begin{align*}
  &K(\alpha)\,{}_2F_1(1-\alpha,-\alpha;-2\alpha;1-u)\\
  &\quad= K(\alpha_0)\,P_k(1-u) \;-\; \frac{\widetilde{C}_k}{\alpha-\alpha_0}(1-u)^m
  \;+\; (1-u)^m\,O(1+|\alpha-\alpha_0|) \;+\; O((1-u)^{m+1}),
\end{align*}
where $P_k$ is a polynomial of degree $m-1$ with $P_k(0)=1$ that arises when $-2\alpha_0$ is a negative integer and $\widetilde{C}_k\in\RR$ is some constant only depending on $k$. In principle, the value of $\widetilde{C}_k$ can be determined by a residue computation. However, we will see below that this is not necessary and we have $\widetilde{C}_k=C_k$ automatically. Also note that we additional used that $K(\alpha)=K(\alpha_0)+O(|\alpha-\alpha_0|)$ to replace the prefactor $K(\alpha)$ of $P_k(1-u)$ by $K(\alpha_0)$.

For the second term on the right-hand side of \eqref{eq:SolAfterTrafo}, we note that $L(\alpha)$ has a simple pole at $\alpha_0$, while ${}_2F_1(2+\alpha,1+\alpha;2(1+\alpha);u)=1+O(1-u)$. Using
\begin{align*}
    \tan(\pi\alpha) &= -\frac{1}{\pi(\alpha-\alpha_0)}+ O(\alpha-\alpha_0),\\
    (1-u)^{1+2\alpha} &= (1-u)^m\,e^{2(\alpha-\alpha_0)\log(1-u)} \\
    &= (1-u)^m\Bigl(1+2(\alpha-\alpha_0)\log(1-u)+O\bigl((\alpha-\alpha_0)^2\log^2(1-u)\bigr)\Bigr),
\end{align*}
we obtain
\begin{equation*}
L(\alpha)(1-u)^{1+2\alpha}
= \frac{C_k}{\alpha-\alpha_0}(1-u)^m \;+\; 2C_k(1-u)^m\log(1-u)
\;+\; (1-u)^m\,O(1+|\alpha-\alpha_0|\,|\log(1-u)|),
\end{equation*}
as $u\to 1^-$ and $\alpha\to\alpha_0$, with
\begin{align*}
C_k={1\over\alpha_0}{\rm Res}_{\alpha=\alpha_0}L(\alpha) = {1\over\alpha_0}\lim_{\alpha\to\alpha_0}(\alpha-\alpha_0)L(\alpha).
\end{align*}
Using \eqref{eq:Lalternative} and the fact that $(\alpha-\alpha_0)\tan(\pi\alpha)\to -\tfrac{1}{\pi}$ as $\alpha\to\alpha_0$, it follows that
$$
{\rm Res}_{\alpha=\alpha_0}L(\alpha) = {\Gamma(\alpha_0)^2\over 4\pi^2(1+2\alpha_0)\Gamma(2\alpha_0)},
$$
and hence
$$
C_k = {\Gamma(\alpha_0)^2\over 4\pi^2\alpha_0(1+2\alpha_0)\Gamma(2\alpha_0)} =\frac{\Gamma\!\bigl(k+\tfrac12\bigr)^2}{8\pi^2\,(k+1/2)(k+1)\,\Gamma(2k+1)},
$$
which proves \eqref{eq:ConstCk}. Adding the two contributions gives, for the left-hand side ${}_2F_1(1-\alpha,-\alpha;2;u)$ in \eqref{eq:SolAfterTrafo},
\begin{align*}
&{}_2F_1(1-\alpha,-\alpha;2;u)
= K(\alpha_0)\,P_k(1-u)
\;+\; \Bigl(-\frac{\widetilde{C}_k}{\alpha-\alpha_0}+\frac{C_k}{\alpha-\alpha_0}\Bigr)(1-u)^m\\
&\qquad\qquad +\; 2C_k(1-u)^m\log(1-u)
\;+\; (1-u)^m\,O(1+|\alpha-\alpha_0|\,|\log(1-u)|)
\;+\; O((1-u)^{m+1}).
\end{align*}
Since ${}_2F_1(1-\alpha_0,-\alpha_0;2;u)$ is finite at $\alpha=\alpha_0$, the simple pole at $\alpha_0$ must cancel, that is, $C_k=\widetilde{C}_k$. Therefore, letting $\alpha\to\alpha_0$, we must have
\begin{equation*}
{}_2F_1(1-\alpha_0,-\alpha_0;2;u)
= K(\alpha_0)\,P_k(1-u) \;+\; 2C_k(1-u)^m\log(1-u) \;+\; (1-u)^m\,O(1) \;+\; O((1-u)^{m+1}).
\end{equation*}
Multiplying this expansion with $(1-u)^{-\alpha_0}$ finally yields
\begin{align*}
F(u,z) 
= K(\alpha_0)(1-u)^{-\alpha_0} + O\bigl((1-u)^{1-\alpha_0}\bigr)
+ 2C_k(1-u)^{1+\alpha_0}\log(1-u) + O\bigl((1-u)^{1+\alpha_0}\bigr),    
\end{align*}
as $u\to 1^-$. This completes the proof of the theorem for half-integers $\alpha=\alpha_0=\tfrac{1}{2}+k$.
\end{proof}

Next, we transfer the singular expansions of $F(u,z)$ to the ``horizontal'' generating function $G_n(z)=[u^n] F(u,z)$, see \cite[p.\ 390]{FS:2009}. 

\begin{corollary}[Asymptotics for the ordinary generating function]\label{cor:AsymptoticsG1}
Fix $k\in\ZZ_{\geq 0}$ and let $I_k=(k-\tfrac12,k+\tfrac12)$. For $\alpha=\alpha(z)\in I_k$ we have
\begin{equation*}
G_n(z)
= \frac{K(\alpha)}{\Gamma(\alpha)}\,n^{\alpha-1}
 + \sum_{j=1}^{k+1} \frac{K(\alpha)a_j(\alpha)}{\Gamma(\alpha-j)}\,n^{\alpha-j-1}
 + \frac{L(\alpha)}{\Gamma(-\alpha-1)}\,n^{-\alpha-2} + O\bigl(n^{\alpha-k-3}\bigr),
\end{equation*}
as $n\to\infty$ uniformly for $\alpha$ in compact subsets of $I_k$, where $K(\alpha),L(\alpha)$ and $a_j(\alpha)$ are defined in \eqref{def:KandL} and \eqref{def:an}.
\end{corollary}
\begin{proof}
For $\alpha>0$ we have
\begin{equation*}
     (1-u)^{-\alpha} 
     = \sum_{n=0}^\infty \frac{\Gamma(\alpha+n)}{\Gamma(\alpha)} \frac{u^n}{n!}
     = 1+\alpha u + \frac{\alpha(\alpha+1)}{2} u^2 + \frac{\alpha(\alpha+1)(\alpha+2)}{3!} u^3 + \ldots
\end{equation*}
and therefore
\begin{equation*}
     [u^n] (1-u)^{-\alpha}
     = \frac{1}{\Gamma(\alpha)} \frac{\Gamma(n+\alpha)}{\Gamma(n+1)}
     = \frac{n^{\alpha-1}}{\Gamma(\alpha)} \left( 1+ \frac{\alpha(\alpha-1)}{2n} + O\!\left(\frac{1}{n^2}\right)\right),
\end{equation*}
by Stirling's formula. We refer to \cite[p.\ 381]{FS:2009} for a complete asymptotic expansion of the coefficients of $(1-u)^{-\alpha}$ for any $\alpha\in\CC\setminus \{0,-1,-2,\dotsc\}$.

For $\alpha\in I_k$, we use the expansion from Theorem \ref{thm:SingExpFuz}:
\begin{equation*}
F(u,z)
= K(\alpha)(1-u)^{-\alpha}
  + \sum_{j=1}^{k+1} K(\alpha)a_j(\alpha)\,(1-u)^{j-\alpha}
  + L(\alpha)(1-u)^{1+\alpha}\bigl(1+O(1-u)\bigr)
  + O\bigl((1-u)^{k+2-\alpha}\bigr),
\end{equation*}
where $a_0(\alpha)=1$, $a_1(\alpha)=\frac{1-\alpha}{2}$, and $a_j(\alpha)$ are analytic in $\alpha$ on compact subsets of $I_k$. By the transfer rules,
\begin{equation*}
G_n(z)
=[u^n] F(u,z)= \frac{K(\alpha)}{\Gamma(\alpha)}\,n^{\alpha-1}
 + \sum_{j=1}^{k+1} \frac{K(\alpha)a_j(\alpha)}{\Gamma(\alpha-j)}\,n^{\alpha-j-1}
 + \frac{L(\alpha)}{\Gamma(-\alpha-1)}\,n^{-\alpha-2} + O(n^{\alpha-k-3})
,
\end{equation*}
as $n\to\infty$ uniformly for $\alpha$ in compact subsets of $I_k$. 
\end{proof}

In the same way, we obtain the following result for half-integer values of $\alpha$.

\begin{corollary}[Asymptotics for the ordinary generating function]\label{cor:AsymptoticsG2}
For $\alpha(z)=\alpha_0=\tfrac12+k$ with $k\in\ZZ_{\geq 0}$ we have
\begin{equation*}
G_n(z)
=\frac{K(\alpha_0)}{\Gamma(\alpha_0)}\,n^{\alpha_0-1}
  + \frac{2C_k(\alpha_0)}{\Gamma(-1-\alpha_0)}\,n^{\alpha_0-2}\log n
  + O(n^{\alpha_0-2}),
\end{equation*}
as $n\to\infty$, where $C_k(\alpha_0)$ is as in Theorem \ref{thm:SingExpFuz}.
\end{corollary}

\section{Probabilistic applications}\label{sec:Applications}

\subsection{Common starting point}\label{sec:StartingPoint}

Fix $t\in\RR$ and write $\alpha:=\alpha(e^t)>0$, where we recall that $\alpha(z)=\tfrac{1}{2}(\sqrt{8z+1}-1)$. By Corollary~\ref{cor:AsymptoticsG1} we have
\begin{equation*}
G_n(e^t)
= \frac{K(\alpha)}{\Gamma(\alpha)}\,n^{\alpha-1}
\Big(1+O(n^{-1})\Big),
\qquad n\to\infty,
\end{equation*}
uniformly for $\alpha$ in compact subsets of $I_k$, where $I_k=(k-\tfrac{1}{2},k+\tfrac{1}{2})$ for some $k\in\ZZ_{\geq 0}$.
On the other hand, by Corollary~\ref{cor:AsymptoticsG2} we have
\begin{equation*}
G_n(e^t)
=\frac{K(\alpha_0)}{\Gamma(\alpha_0)}\,n^{\alpha_0-1}
  + \frac{2C_k}{\Gamma(-1-\alpha_0)}\,n^{\alpha_0-2}\log n
  + O\bigl(n^{\alpha_0-2}\bigr),
\end{equation*}
for $\alpha=\alpha_0=\tfrac{1}{2}+k$ and some $k\in\ZZ_{\geq 0}$. So in both cases
\begin{equation*}
\log G_n(e^t)
= (\alpha(e^t)-1)\log n + \log\!\frac{K(\alpha(e^t))}{\Gamma(\alpha(e^t))} + o(1) \qquad \text{as $n\to\infty$}.
\end{equation*}
It follows that
\begin{equation}\label{eq:GE-limit}
\Lambda_n(t):=\frac{1}{\log n}\log \EE e^{tf_0(T_n)} = \frac{1}{\log n} \sum_{k=0}^n p_k^{(n)} e^{t(k+2)} = \frac{2t+\log G_n(e^t)}{\log n}
= \mu(t) + \frac{\psi(t)}{\log n}+o(1)
\end{equation}
with
$$
\mu(t):=\alpha(e^t)-1,\qquad
\psi(t):=2t + \log\!\frac{K(\alpha(e^t))}{\Gamma(\alpha(e^t))}
$$
uniformly for $t$ in a neighborhood of $0$. The asymptotic expansion \eqref{eq:GE-limit} will be the common starting point for all probabilistic results we present in the following sections.

\subsection{Precise cumulant asymptotics}

We start with the following cumulant asymptotics, which is a direct consequence of \eqref{eq:GE-limit}. For completeness, let us recall that for an integer $r\geq 1$ the $r$-th order cumulant of a random variable $X$ whose moment generating function $G(t):=\EE e^{tX}$ is finite in a neighborhood of $0$ is defined as
$$
\kappa(X) := {\dint^r\over\dint t^r}\log G(t)\Big|_{t=0} = r![t^n]\log G(t).
$$

\begin{theorem}[Cumulant asymptotics]\label{thm:CumulantAsymptotics}
For integers $r\geq 1$ let $\kappa^{(r)}_n$ be the $r$-th order cumulant of the random variable $f_0(T_n)$. Then
$$
\kappa_n^{(r)} = \mu^{(r)}(0)(\log n)\,(1+o(1))
$$
as $n\to\infty$, where $\mu^{(r)}(0)$ is the $r$-th order derivative of $\mu(t)$ at $t=0$.
\end{theorem}
\begin{proof}
From \eqref{eq:GE-limit} we have
\begin{equation*}
\log \EE\big[e^{t f_0(T_n)}\big]=\mu(t)\,\log n+\psi(t)+o(1) 
\end{equation*}
uniformly in a neighborhood of $0$. We may thus differentiate $r$ times termwise and evaluate at $t=0$:
\begin{equation*}
\frac{\dint^{r}}{\dint t^{r}}\log \EE\big[e^{tX_n}\big]\Big|_{t=0}
=\mu^{(r)}(0)\,\log n + \psi^{(r)}(0) + o(1).
\end{equation*}
By definition, the left-hand side equals the $r$-th cumulant $\kappa_n^{(r)}$ of the random variable $f_0(T_n)$. This proves the claim.
\end{proof}

We shall now make the prefactors $\mu^{(r)}(0)$ fully explicit in terms of the Stirling numbers of the second kind $S(r,k)$ for which we refer to \cite[Section 6.1]{ConcreteMathematics}.

\begin{lemma}[Closed form for the prefactor]\label{lem:Prefactor}
Consider the function
\begin{equation*}
\alpha(t):=\frac{\sqrt{\,1+8e^{t}\,}-1}{2},\qquad t\in\RR.
\end{equation*}
Then, for every integer $r\ge 1$, the $r$-th derivative of $\alpha(t)$ at $t=0$ satisfies
\begin{equation*}
\alpha^{(r)}(0)
=\sum_{k=1}^{r} (-1)^{k-1}S(r,k)
\frac{2^{k}\,(2k-2)!}{3^{2k-1}\,(k-1)!}.
\end{equation*}
\end{lemma}
\begin{proof}
Set $y=e^{t}$ and $f(y):=\dfrac{\sqrt{1+8y}-1}{2}$, so $\alpha(t)=f(e^{t})$. The identity
\begin{equation*}
\Big(y\frac{\dint}{\dint y}\Big)^{r}=\sum_{k=1}^{r} S(r,k)\,y^{k}\frac{\dint^{k}}{\dint y^{k}}
\end{equation*}
from \cite[p.\ 218]{Riordan1968} or \cite[Ex.\ 13 on p.\ 310]{ConcreteMathematics}
gives, at $t=0$ (or equivalently $y=1$),
\begin{equation*}
\alpha^{(r)}(0)=\frac{\dint^{r}}{\dint t^{\,r}}f(e^{t})\bigg|_{t=0}
=\sum_{k=1}^{r} S(r,k)\,f^{(k)}(1).
\end{equation*}
Writing $f(y)=\tfrac{1}{2}(1+8y)^{1/2}-\tfrac{1}{2}$ and taking the $k$-th order derivative, we arrive at
\begin{equation*}
f^{(k)}(y)=\frac{1}{2}\cdot(1/2)^{\underline{k}}\;8^{k}\,(1+8y)^{1/2-k},
\end{equation*}
where $(q)^{\underline{k}}=q(q-1)\cdots(q-k+1)$ is the falling factorial. Now, take $y=1$:
\begin{equation*}
f^{(k)}(1)=\frac{1}{2}\,(1/2)^{\underline{k}}\,8^{k}\,9^{1/2-k}
=\frac{1}{2}\,(1/2)^{\underline{k}}\,8^{k}\,3^{1-2k}.
\end{equation*}
Using that
\begin{equation*}
(1/2)^{\underline{k}}
= (-1)^{k-1}{(2k-3)!!\over 2^k} = (-1)^{k-1}{(2k-2)!\over 2^{2k-1}(k-1)!}
\end{equation*}
we obtain
\begin{equation*}
f^{(k)}(1)=(-1)^{k-1}\,\frac{1}{2}\,\frac{(2k-2)!}{2^{2k-1}(k-1)!}\,8^k\,3^{\,1-2k}
=(-1)^{k-1}\,\frac{2^{\,k}(2k-2)!}{3^{\,2k-1}(k-1)!}.
\end{equation*}
Plugging this back into the expression for $\alpha^{(r)}(0)$ proves the claim.
\end{proof}

A combination of Theorem \ref{thm:CumulantAsymptotics} with Lemma \ref{lem:Prefactor} yields the following explicit cumulant asymptotics.

\begin{corollary}[Explicit cumulant asymptotics]\label{cor:ExplicitCumulant}
For integers $r\geq 1$ we have
$$
\kappa_n^{(r)} = \Bigg(\sum_{k=1}^{r} (-1)^{k-1}S(r,k)
\frac{2^{k}\,(2k-2)!}{3^{2k-1}\,(k-1)!}\Bigg)\log n\,(1+o(1))
$$
as $n\to\infty$.
\end{corollary}
\begin{proof}
    This is a direct consequence of Theorem \ref{thm:CumulantAsymptotics} and Lemma \ref{lem:Prefactor}, since $\mu^{(r)}(0)=\alpha^{(r)}(0)$.
\end{proof}

In particular, for $r\in\{1,2,3,4,5,6\}$ we obtain
\begin{equation*}
\begin{aligned}
\kappa_n^{(1)}&=\frac{2}{3}\,(\log n)\,(1+o(1)), &\quad
\kappa_n^{(2)}&=\frac{10}{27}\,(\log n)\,(1+o(1)), &\quad
\kappa_n^{(3)}&=\frac{14}{81}\,(\log n)\,(1+o(1)),\\[2mm]
\kappa_n^{(4)}&=\frac{62}{729}\,(\log n)\,(1+o(1)), &\quad
\kappa_n^{(5)}&=\frac{334}{6561}\,(\log n)\,(1+o(1)), &\quad
\kappa_n^{(6)}&=\frac{110}{6561}\,(\log n)\,(1+o(1)).
\end{aligned}
\end{equation*}
which is consistent with Corollaries 1--4 in \cite{Buchta:2012} for $r\in\{1,2,3,4\}$. In this context, Corollary \ref{cor:ExplicitCumulant} can be seen as an analog for cumulants of the simple moment asymptotics
$$
\EE(f_0(T_n)^r) = \Big(\frac{2}{3}\log n\Big)^r\,(1+o(1))
$$
provided by \cite[Theorem 3]{Buchta:2012}. From the asymptotic behavior of the first six cumulants, it may seem that the leading coefficients in front of $\log n$ is always positive. However, this is true only until $\kappa_n^{(8)}$, since starting with $\kappa_n^{(9)}=-\tfrac{52598}{1594323}(\log n)\,(1+o(1))$ we see that the coefficients of the leading $(\log n)$-term can also become negative.

\subsection{Asymptotic normality}

We shall now use \eqref{eq:GE-limit} to derive a central limit theorem for the sequence of random variables $f_0(T_n)$ together with a rate of convergence. We note that the same result has already been derived in \cite{GT:2021} by a completely different method.

In what follows, we write $\Phi(x):={1\over\sqrt{2\pi}}\int_{-\infty}^x e^{-y^2/2}\,\dint y$, $x\in\RR$, for the distribution function of a standard Gaussian random variable.

\begin{theorem}[Quantitative central limit theorem]\label{thm:CLT-rate}
The sequence of normalized random variables
\begin{equation*}
    Z_n:=\frac{f_0(T_n)- \EE f_0(T_n)}{\sqrt{\var(f_0(T_n))}},\qquad n\geq 0
\end{equation*}
satisfies 
\begin{equation*}
    \sup_{x\in\RR}\bigl|\PP(Z_n\le x)-\Phi(x)\bigr|=O\Big(\frac{1}{\sqrt{\log n}}\Big).
\end{equation*}
\end{theorem}
\begin{proof}
We apply Hwang's quasi-power method, see \cite[Section IX.5]{FS:2009}. This result asserts that if the cumulant
generating function of a sequence of random variables $X_n$ has the form
$\log \EE e^{tX_n}=\lambda_n u(t)+v(t)+o(1)$ uniformly near $t=0$ with
$u''(0)>0$ and $\lambda_n\to\infty$, then
\begin{equation*}
    \sup_{x\in\RR}\Big|\PP\Big(\frac{X_n-\EE X_n}{\sqrt{\var X_n}}\le x\Big)-\Phi(x)\Big|
        = O\big(\lambda_n^{-1/2}\big).
\end{equation*}
In our case, this is satisfied with $X_n=f_0(T_n)$, $u(t)=\mu(t)$, $v(t)=\psi(t)$, and $\lambda_n=\log n$. Moreover, $u''(0)=\tfrac{10}{27}$ according to Corollary \ref{cor:ExplicitCumulant}. This proves the claim. 
\end{proof}

%{\color{red}CT: Ich könnte gut damit leben, wenn wir die nächste Proposition einfach wegnehmen würden.}
%{\color{blue} FB: Von mir aus können wir es auch raus nehmen, aber ich denke sparen uns auch nur eine halbe Seite, also nicht so viel...}

Using \eqref{eq:GE-limit} it is also possible to derive higher-order corrections to the Gaussian limit law, known as Edgeworth expansions.

\begin{proposition}[First-order Edgeworth expansion]
For $y\in\RR$ let $\varphi(y):=e^{-y^2/2}/\sqrt{2\pi}$ be the standard Gaussian density and $\mathcal{S}(y):=\tfrac{1}{2}-\{y\}$ be the centered sawtooth function.
Then, for the normalized random variables defined as in Theorem \ref{thm:CLT-rate},
$$
\PP(Z_n\leq x) = \Phi(x) + {7\sqrt{30}\over 300}{1\over\sqrt{\log n}}(1-x^2)\varphi(y)+{1\over 2\sqrt{\log n}}\sqrt{27\over 10}\mathcal{S}(\EE f_0(T_n)-x\sigma_n)\varphi(x)+o(1/\sqrt{\log n})
$$
uniformly for $x\in\RR$, where $\sigma_n:=\sqrt{\kappa_n^{(2)}}$.
\end{proposition}
\begin{proof}
The first-order Edgeworth expansion for non-negative integer-valued random variables says that
$$
\PP(Z_n\leq x) = \Phi(x) + \frac{\gamma_n^{(3)}}{6}(1-x^2)\varphi(x) + {1\over\sigma_n}\mathcal{S}(\EE f_0(T_n)-x\sigma_n)\varphi(x) + o(1/\sigma_n + \gamma_n^{(3)})
$$
uniformly in $x\in\RR$, where $\gamma_n^{(3)}:=\kappa_n^{(3)}/(\log n)^{3/2}$ is the third standardized cumulant of $f_0(T_n)$, $\varphi(y)=\tfrac{1}{\sqrt{2\pi}}e^{-y^2/2}$, $y\in\RR$, see \cite[Chapter 23]{BhattRao}. In our case, 
\begin{align*}
    \kappa_n^{(2)} &= \frac{10}{27}(\log n)(1+o(1)),\\
    \gamma_n^{(3)} &= {\kappa_n^{(3)}\over(\kappa_n^{(2)})^{3/2}} = {{14\over 81}(\log n)(1+o(1))\over\big({10\over 27}(\log n)(1+o(1))\big)^{3/2}} = {14\over 81}\Big({27\over 10}\Big)^{3/2} {1\over\sqrt{\log n}}(1+o(1))= {7\sqrt{30}\over 50}{1\over\sqrt{\log n}}(1+o(1)),
\end{align*}
and we obtain the result, since $|S(y)|\leq 1/2$ for all $y\in\RR$.
\end{proof}

\begin{remark}
    Higher-order expansions can be deduced in the same way.
\end{remark}

\subsection{Moderate and large deviation principle, laws of large numbers}\label{sec:MDPandLDP}

The next results we present are a moderate and a large deviation principle for $f_0(T_n)$, the latter of which was not within reach of the previously existing methods. Let us recall that a sequence of random variables $X_n$ satisfies a large deviation principle on $\RR$ with speed $s_n$ and a good rate function $I:[0,\infty)\to[-\infty,+\infty]$ if
\begin{itemize}
    \item[-] $s_n\geq 0$ with $s_n\to\infty$,
    \item[-] $I$ is lower semicontinuous and has compact level sets $\{x\in\RR:I(x)\leq c\}$ for any $c\geq 0$,
    \item[-] for every Borel set $B\subseteq\RR$ we have
    \begin{align*}
        -\inf_{x\in{\rm int}(B)}I(x) \leq \liminf_{n\to\infty}{1\over s_n}\log\PP(X_n\in B) \leq \limsup_{n\to\infty}{1\over s_n}\log\PP(X_n\in B) \leq -\inf_{x\in{\rm cl(B)}}I(x),
    \end{align*}
    where ${\rm int}(B)$ and ${\rm cl(B)}$ denote the interior and the closure of $B$, respectively.
\end{itemize}
Roughly speaking, if the random variables $X_n$ satisfy a large deviation principle as above, we have that the probabilities of rare events decay exponentially at speed $s_n$ and the rate function $I$ quantifies this decay. In particular,
$$
\PP(X_n\in B) \approx e^{-s_n I(B)},\qquad I(B):=\inf_{x\in B}I(x),
$$
for well-behaved Borel sets $B\subseteq\RR$ as $n\to\infty$. For background material on large deviation theory, we refer to \cite{DZlargeDeviations,DenHollander}.

We start with a large deviation principle on scales between $\sqrt{\log n}$ and $\log n$, which is also called a moderate deviation principle. The resulting rate function $I(x)=x^2/2$ reflects the exponent of the standard Gaussian distribution appearing in the central limit theorem. We note that a moderate deviation principle for $f_0(T_n)$ was obtained in \cite{GT:2021} by entirely different methods.

\begin{theorem}[Moderate deviation principle]\label{thm:MDP}
Let $\{a_n\}_{n\in\NN}$ be any positive sequence with
\begin{equation*}
a_n\to\infty\qquad\text{and}\qquad \frac{a_n}{\sqrt{\log n}}\to 0.
\end{equation*}
Then the rescaled variables
\begin{equation*}
Y_n:=\frac{f_0(T_n)-{2\over 3}\log n}{\sqrt{10\over 27}\,{a_n\sqrt{\log n}}}
\end{equation*}
satisfy a large deviation principle on $\RR$ with speed $a_n^2$ and a good rate function $I(x)=x^2/2$.
\end{theorem}
\begin{proof}
Fix $s\in\RR$ and put $m:=2/3$, $\sigma^2:=10/27$ and
\begin{equation*}
t_n:=\frac{s\,a_n}{\sigma\sqrt{\log n}}\to 0
\end{equation*}
as $n\to\infty$.
Consider the scaled cumulant generating function of $Y_n$:
\begin{equation*}
\Lambda_n(s):=\frac{1}{a_n^2}\log \EE\exp\Big(s\,a_n\,\frac{f_0(T_n)-m\log n}{\sigma\sqrt{\log n}}\Big)
=\frac{1}{a_n^2}\Big(\log \EE e^{t_n f_0(T_n)}-t_n m\log n\Big).
\end{equation*}
By \eqref{eq:GE-limit},
\begin{equation*}
\log \EE e^{t_n f_0(T_n)}=\mu(t_n)\,\log n+\psi(t_n)+o(1),
\end{equation*}
hence
\begin{equation*}
\Lambda_n(s)=\frac{\log n}{a_n^2}\,\big(\mu(t_n)-m\,t_n\big)
+\frac{\psi(t_n)}{a_n^2}+o(1/a_n^2).
\end{equation*}
Using that $\mu(t)$ is analytic around $t=0$, we can make a Taylor expansion, which yields
\begin{equation*}
\mu(t)=m t+\frac{\sigma^2}{2}t^2+\frac{\mu^{(3)}(0)}{6}t^3+O(t^4).
\end{equation*}
Therefore,
\begin{align*}
\frac{\log n}{a_n^2}\big(\mu(t_n)-m t_n\big)
&=\frac{\log n}{a_n^2}\left(\frac{\sigma^2}{2}\frac{s^2 a_n^2}{\sigma^2 \log n}
+\frac{\mu^{(3)}(0)}{6}\frac{s^3 a_n^3}{\sigma^3 (\log n)^{3/2}}
+O\Big(\frac{a_n^4}{(\log n)^2}\Big)\right)\\
&=\frac{s^2}{2}\;+\; \frac{\mu^{(3)}(0)}{6\sigma^3}\,\frac{s^3\,a_n}{\sqrt{\log n}}
\;+\;O\Big(\frac{a_n^2}{\log n}\Big).
\end{align*}
Since $a_n/\sqrt{\log n}\to 0$, both error terms tend to $0$ uniformly on compact sets. Moreover,
$\psi(t_n)=\psi(0)+O(t_n)=\psi(0)+O(a_n/\sqrt{\log n})$, so $\psi(t_n)/a_n^2\to 0$. Hence,
\begin{equation*}
\lim_{n\to\infty}\Lambda_n(s) = \Lambda(s):=\frac{s^2}{2},
\end{equation*}
pointwise for all $s\in\RR$. The limiting function $\Lambda$ is  finite everywhere, differentiable, and strictly convex. This allows us to apply the Gärtner--Ellis theorem \cite[Theorem 2.3.6]{DZlargeDeviations} to conclude that the sequence of random variables $Y_n$ satisfies an LDP with speed $a_n^2$ and good rate function $\Lambda^*$, where
$$
\Lambda^{*}(x) = \sup_{t\in\RR}[xt-\Lambda(t)]=\sup_{t\in\RR}\Big[xt-{t^2\over 2}\Big],\qquad s\in\RR,
$$
denotes the Legendre--Fenchel transform of $\Lambda$. Since ${\dint\over\dint t}[xt-{t^2\over 2}]=x-t=0$ if and only if $t=x$, it follows that $\Lambda^{*}(x)=x^2/2$, as claimed.
\end{proof}

Next, we turn to the case $a_n=\sqrt{\log n}$ not covered by Theorem \ref{thm:MDP}. We will see that the random variables $f_0(T_n)/\log n$ still satisfy a large deviation principle, but the rate function, illustrated in Figure \ref{fig:LDP}, is fundamentally different from the Gaussian rate function in Theorem \ref{thm:MDP}.

\begin{figure}[t]
\centering
\begin{tikzpicture}
  \begin{axis}[
      width=10cm, height=6cm,
      domain=0.001:3,
      samples=400,
      axis lines=middle,
      xmin=0, xmax=3,
      ymin=-0.1, ymax=4%,
      %grid=both%,
      %legend style={at={(0.97,0.97)},anchor=north east}
    ]
    % Plot of I(x)
    \addplot[blue, thick] 
      {x*ln(x/2) + x*ln(2*x + sqrt(1+4*x^2)) - x - 0.5*sqrt(1+4*x^2) + 1.5};
    %\addlegendentry{$I(x)$}

    % Vertical line at x=2/3
    \addplot[dashed, red, thick] coordinates {(2/3,-0.1) (2/3,4)};
    \node[red, anchor=north east] at (axis cs:2/3,0) {$\tfrac{2}{3}$};
  \end{axis}
\end{tikzpicture}
\caption{The rate function $I(x)$ for $x\geq 0$ of the large deviation principle in Theorem \ref{thm:LDP}. It is strictly convex with a unique minimum of $0$ at $x=2/3$.}
\label{fig:LDP}
\end{figure}
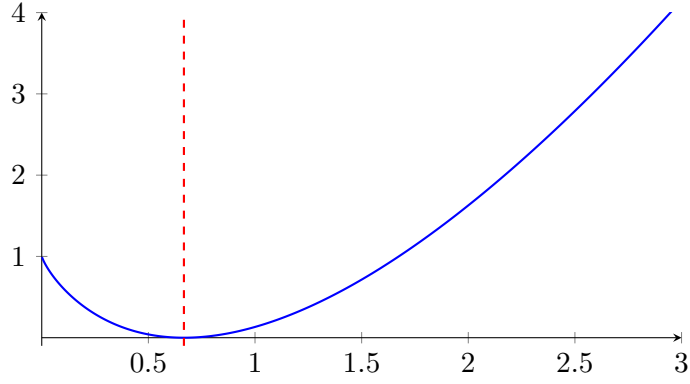

\begin{theorem}[Large deviation principle]\label{thm:LDP}
The sequence of random variables $\tfrac{f_0(T_n)}{\log n}$ satisfies a large deviation principle on $\RR$ with speed $\log n$ and good rate function
\begin{equation}\label{eq:rate-explicit}
I(x)=\begin{cases}
    + \infty &: x<0\\
    x\log\!\Big(\frac{x}{2}\Big)+x\,\operatorname{arsinh}(2x)\,-\,x\,-\,\frac12\sqrt{1+4x^2}\,+\,\frac32 &: x\geq 0.
\end{cases}
\end{equation}
\end{theorem}
\begin{proof}
From \eqref{eq:GE-limit} we conclude that $\Lambda_n(t)\to\mu(t)$ pointwise on $\RR$ as $n\to\infty$. From the explicit formula for $\mu$ we compute
\begin{equation*}
\mu'(t)=\frac{2e^t}{\sqrt{8e^t+1}}>0,\qquad
\mu''(t)=\frac{2e^t(8e^t+2)}{(8e^t+1)^{3/2}}>0,
\end{equation*}
so $\mu$ is analytic, strictly convex, and finite everywhere on $\RR$.
%The domain of $\mu$ is the whole real line, so no steepness condition needs to be checked.

By the Gärtner--Ellis theorem \cite[Theorem 2.3.6]{DZlargeDeviations}, the sequence of random variables $f_0(T_n)/\log n$ satisfies a large deviation principle on $\RR$ with speed $\log n$ and good rate function $I=\mu^{*}$, where again
\begin{equation}\label{eqn:LF-Irate}
I(x)=\mu^{*}(x) = \sup_{t\in\RR}[xt-\mu(t)],\qquad x\in\RR,
\end{equation}
denotes the Legendre--Fenchel transform of $\mu$.

Since $\mu'(t)$ is strictly increasing with range $(0,\infty)$, the effective domain $\{x\in\RR:I(x)<\infty\}$ of $I$ is $[0,\infty)$, and $I(x)=+\infty$ for $x<0$. For $x>0$ the maximizer in the Legendre--Fenchel transform is the unique $t=t(x)$ with $\mu'(t)=x$. Writing $y=e^{t(x)}$ and solving
\begin{equation*}
x=\frac{2y}{\sqrt{8y+1}}
\quad\Longleftrightarrow\quad
y=x^2+\frac{x}{2}\sqrt{1+4x^2},
\end{equation*}
one obtains
\begin{equation*}
I(x)=xt(x)-\mu\bigl(t(x)\bigr)
= x\log y-\frac{\sqrt{8y+1}-3}{2}.
\end{equation*}
Since $\sqrt{8y+1} = \tfrac{2y}{x}$, this becomes
\begin{equation*}
I(x) = x\log y - \frac{y}{x} + \frac{3}{2}.
\end{equation*}
Factoring $y$ gives
\begin{equation*}
y = \frac{x}{2}\Bigl(2x+\sqrt{1+4x^2}\Bigr),
\end{equation*}
so that
\begin{equation*}
\log y = \log\Bigl(\tfrac{x}{2}\Bigr) + \log\bigl(2x+\sqrt{1+4x^2}\bigr)
= \log\Bigl(\tfrac{x}{2}\Bigr) + \arsinh(2x).
\end{equation*}
Hence
\begin{equation*}
x\log y = x\log\!\Bigl(\tfrac{x}{2}\Bigr) + x\,\arsinh(2x).
\end{equation*}
Moreover,
\begin{equation*}
\frac{y}{x} = x + \tfrac{1}{2}\sqrt{1+4x^2}.
\end{equation*}
Combining these identities yields
\begin{equation*}
I(x) \;=\; x\log\Bigl(\tfrac{x}{2}\Bigr) + x\,\arsinh(2x) \;-\; x 
\;-\; \tfrac{1}{2}\sqrt{1+4x^2} + \tfrac{3}{2}.
\end{equation*}
Together with $I(x)=+\infty$ for $x<0$, this proves \eqref{eq:rate-explicit}. 
\end{proof}

Parts of the rate function $I(x)$ appearing in Theorem \ref{thm:LDP} admit a clear probabilistic interpretation. To develop it, we decompose $I(x)$ as follows:
$$
I(x) = \Big[x\log\Bigl(\tfrac{x}{2}\Bigr)-x+2\Big] + \Big[x\,\arsinh(2x)-\tfrac{1}{2}\sqrt{1+4x^2} -\tfrac{1}{2}\Big] =: I_1(x) + I_2(x).
$$
We extend $I_1$ to the negative real axis by putting $I_1(x):=+\infty$, $I_2$ is well-defined for negative arguments. Then, using the Gärtner--Ellis theorem, the following can be checked: The function $I_1(x)$ is exactly the good rate function of a large deviation principle for the random variables $\tfrac{X_n}{\log n}$, where $X_n$ is Poisson distributed with parameter $2\log n$.\\ Indeed, for all $n\in\NN$,
$$
{1\over\log n}\log\EE e^{tX_n} = {1\over\log n}\big(2\log n(e^t-1)\big) = 2(e^t-1),\qquad t\in\RR,
$$
and the Legendre--Fenchel transform of this function is precisely $I_1(x)$. However, since $I_2$ takes negative values, it cannot be a rate function. On the other hand, the function $J_2(x):=I_2(x)+1$, $x\in\RR$, again has a probabilistic meaning. It is the good rate function of a large deviation principle for the random variables $\tfrac{Y_n}{\log n}$ with $Y_n:=Y_n^{(1)}-Y_n^{(2)}$ and where $Y_n^{(1)}$ and $Y_n^{(2)}$ are independent and Poisson distributed with parameter $\tfrac{1}{4}\log n$. Equivalently, $Y_n=\sum_{k=1}^{N_n}\varepsilon_k$, where $N_n$ is Poisson distributed with parameter $\tfrac{1}{2}\log n$ and $(\varepsilon_k)$ is a sequence of independent Rademacher random variables taking values $\pm 1$ with probability $1/2$, which is also independent of $N_n$. Indeed, for $n\in\NN$ we have
    $$
    \log\EE e^{tY_n} = {\log n\over 4}(e^t-1) +{\log n\over 4}(e^{-t}-1) = {\log n\over 2}(\cosh(t)-1),
    $$
    and hence
    $$
    {1\over\log n}\log\EE e^{tY_n} = {1\over 2}(\cosh(t) -1),\qquad t\in\RR.
    $$
    The Legendre--Fenchel transform of this function is precisely $J_2(x)$.

\medskip

%{\color{red}CT: Ich könnte gut damit leben, die nächsten beiden Aussagen wegzunehmen.}
%{\color{blue} FB: Ich überlasse die endgültige Entscheidung dir. Ich denke es spart uns rund eine Seite, aber bei nun fast 30 Seiten insgesamt fällt das auch kaum noch ins Gewicht.}

The rate function $I(x)$ in Theorem \ref{thm:LDP} is coercive on $[0,\infty)$, which means that $I(x)\to\infty$ as $x\to\infty$. Moreover, $I$ is continuous and attains its unique minimum $0$ at $x=\mu'(0)=2/3$, see Figure \ref{fig:LDP}. These two facts already implies a weak law of large numbers:

\begin{corollary}[Weak law of large numbers]\label{cor:WLLN}
    We have
    $$
   \frac{f_0(T_n)}{\log n}\to{2\over 3}\qquad\text{in probability as}\ n\to\infty.
   $$
\end{corollary}
\begin{proof}
   Indeed, for $\varepsilon>0$ and $n\in\NN$ write $B_{n,\varepsilon}$ for the event that $|\tfrac{f_0(T_n)}{\log n}-\tfrac{2}{3}|>\varepsilon$. The large deviation principle in Theorem \ref{thm:LDP} applied to $B=B_{n,\varepsilon}$ yields
   $$
   \limsup_{n\to\infty}{1\over\log n}\log\PP(B_{n,\varepsilon}) \leq -\inf_{|x-2/3|>\varepsilon}I(x).
   $$
   Since $I$ is continuous, coercive and attains its unique minimum at $2/3$, $\inf_{|x-2/3|>\varepsilon}I(x)=:\delta>0$. It follows that
   \begin{equation}\label{eq:WLLN}
   \PP(B_{n,\varepsilon}) \leq e^{-\delta\log n + o(\log n)} = n^{-\delta + o(1)} \to 0,  
   \end{equation}
   as required for the weak law of large numbers. 
\end{proof}

Our next result shows that the weak law of large numbers can be upgraded to a strong law.

\begin{proposition}[Strong law of large numbers]
    We have
    $$
    {f_0(T_n)\over\log n}\to\frac{2}{3}\qquad\text{almost surely.}
    $$
\end{proposition}
\begin{proof}
    In the proof of Corollary \ref{cor:WLLN} we have already seen that for all $\varepsilon>0$ we can find some $N\in\NN$ such that for all $n\in\NN$, $\PP(B_{n,\varepsilon})\leq n^{-\delta/2}$, say, where $\delta=\delta(\varepsilon)$. Note that the factor $1/2$ in the exponent absorbs the $o(\log n)$ term in \eqref{eq:WLLN}; any other real number in $(0,1)$ can also be used instead. Let $n_k:=[e^{k^2}]$ and apply the previous bound to $n_k$ instead of $n$. Since $e^{-\delta k^2/2}$ is summable with respect to $k$, the Borel--Cantelli lemma implies that
    $$
    \frac{f_0(T_{n_k})}{\log n_k}\to {2\over 3}\qquad\text{almost surely},
    $$
    which is the strong law of large numbers along the subsequence $n_k$. 
    
    For an arbitrary sequence $\varepsilon_k\to 0$ define the block $[n_k,n_{k+1}):=\{n\in\NN:n_k\leq n<n_{k+1}\}$ and the event $B_k:=\{|f_0(T_n)/\log n-2/3|>\varepsilon_k\text{ for some }n\in[n_k,n_{k+1})\}$. The union bound implies 
    $$
    \PP(B_k) \leq \sum_{n=n_k}^{n_{k+1}-1}\PP\Big(\Big|{f_0(T_n)\over\log n}-\frac{2}{3}\Big|\geq\varepsilon_k\Big) \leq (n_{k+1}-n_k)n_k^{-\delta/2},\qquad \delta=\delta(\varepsilon_k).
    $$
    Since $n_{k+1}/n_k = e^{2k+1}$ we have $n_{k+1}-n_k \leq Ce^{2k}n_k$ for some finite constant $C>0$. Thus,
    $$
    \PP(B_k) \leq Ce^{2k-(1-\delta/2)k^2},\qquad\delta=\delta(\varepsilon_k)
    $$
    which is summable with respect to $k$. Applying the Borel--Cantelli lemma once again yields that almost surely
    $$
    \max_{n\in[n_k,n_{k+1})} \Big|{f_0(T_n)\over\log n}-\frac{2}{3}\Big| \leq \varepsilon_k
    $$
    for large enough $k$. Together with the already established strong law of large numbers for the subsequence $n_k$, this proves the result.
\end{proof}

\subsection{Probability asymptotic}

Finally, we trace back to the probabilities $p_k^{(n)}=\PP(f_0(T_n)=k+2)$ which were our starting point in constructing the bivariate generating function that extends to the analytic function $F(u,z)$ on $\CC^2\setminus [1,\infty)\times (-\infty,-1/8]$.

From \eqref{eq:pnk} for $k\in\{1,2\}$ we derive
\begin{equation*}
    p_1^{(n)} = \frac{2}{n+1} \qquad \text{and} \qquad
    p_2^{(n)} = \sum_{j=1}^{n-1} \frac{j}{\binom{j+1}{2}} \frac{n-j}{\binom{n+1}{2}} = \frac{4}{n}\left[ \left( \sum_{j=1}^{n-1} \frac{1}{j+1}\right) - 1 + \frac{2}{n+1}\right] \sim 4\frac{\log n}{n},
\end{equation*}
where we write $a_n\sim b_n$ for two sequences $\{a_n\}_{n\in\NN}$ and $\{b_n\}_{n\in\NN}$ if $a_n/b_n\to 1$, as $n\to\infty$.
Furthermore, 
\begin{equation}\label{eqn:pnn1}
    p_n^{(n)} = \frac{2^n}{n!(n+1)!} \qquad \text{and}\qquad
    p_{n-1}^{(n)} = \frac{2^n}{n!(n+1)!} \sum_{j=1}^{n-1} j(j+1) = \frac{2^n n^2}{3 (n!)^2} \left(1-\frac{1}{n}\right).
\end{equation}

However, for $k\geq 3$, directly determining the asymptotic behavior of $p_k^{(n)}$ or $p_{n-k+1}^{(n)}$ as $n\to\infty$ from the explicit formula \eqref{eq:pnk} seems infeasible. Our analysis of the singularities of $F(u,z)$ easily yields asymptotic estimates for the probabilities $p_k^{(n)} = [u^n z^k] F(u,z)$.

\begin{theorem}[Probability asymptotic for fixed $k$]\label{thm:probab}
    For $k\in\ZZ_{\geq 0}$ we have
    \begin{equation*}
        p_{k+1}^{(n)} = \frac{2^{k+1}}{k!} \frac{(\log n)^{k}}{n} (1+O(1/\log n)),
    \end{equation*}
    and
    \begin{equation}\label{eqn:pnnk}
        p_{n-k}^{(n)} = \frac{n^{3k-1}}{3^k k!} \frac{2^n}{(n!)^2} (1+O(1/n)),
    \end{equation}
    as $n\to\infty$.
\end{theorem}
\begin{proof}
    As mentioned above, by Corollary \ref{cor:AsymptoticsG1}, we find
    \begin{equation*}
        [u^n] F(u,z) = G_n(z) = \frac{K(\alpha)}{\Gamma(\alpha)} n^{\alpha-1} (1+O(1/n))
        \quad \text{for $n\to \infty$},
    \end{equation*}
    uniformly for $\alpha=\alpha(z)$ in compact subsets of $(-1/2,1/2)$, that is, for $z$ in compact subsets of $(-1/8,3/8)$.

    Note that $\alpha(z) = \frac{1}{2}(\sqrt{8z+1}-1)$ gives
    \begin{equation*}
        z = \frac{\alpha(\alpha+1)}{2} = \frac{\alpha}{\varphi(\alpha)},
    \end{equation*}
    where $\varphi(\alpha):=\frac{2}{1+\alpha}$.
    Thus, by the Lagrange--Bürmann formula, see \cite[Section 7.32]{WiWat}, for any analytic function $H$ we have
    \begin{equation*}
        [z^k] H(\alpha(z)) = \frac{1}{k} [\alpha^{k-1}] H'(\alpha) \varphi(\alpha)^k = [\alpha^k] H(\alpha) \varphi(\alpha)^{k-1}(\varphi(\alpha)-\alpha\varphi'(\alpha)).
    \end{equation*}
    From this we derive
    \begin{align*}
        [z^k] \frac{K(\alpha)}{\Gamma(\alpha)} n^{\alpha-1} 
            &= [\alpha^{k}] \frac{K(\alpha)}{\Gamma(\alpha)} n^{\alpha-1} \frac{2^{k-1}}{(\alpha+1)^{k-1}} 
                \left(\frac{2}{\alpha+1} + \frac{2\alpha}{(\alpha+1)^2}\right)\\
            &= \frac{2^k}{n} [\alpha^k] \frac{K(\alpha)}{\Gamma(\alpha)} \frac{2\alpha+1}{(\alpha+1)^{k+1}} e^{\alpha \log n}\\
            &= \frac{2^k}{n} [\alpha^k] \frac{K(\alpha)}{\Gamma(\alpha)} \frac{2\alpha+1}{(\alpha+1)^{k+1}} \sum_{j=0}^\infty \frac{(\log n)^j}{j!} \alpha^j.
    \end{align*}
    Since 
    \begin{equation*}
        \frac{K(\alpha)}{\Gamma(\alpha)} \frac{2\alpha+1}{(\alpha+1)^{k+1}} = \alpha + O(\alpha^2),
    \end{equation*}
    is analytic for $|\alpha|<1/2$,
    we conclude
    \begin{equation*}
        [z^k] \frac{K(\alpha)}{\Gamma(\alpha)} n^{\alpha-1} = \frac{2^k}{n} \frac{(\log n)^{k-1}}{(k-1)!} + O\left(\frac{(\log n)^{k-2}}{n}\right).
    \end{equation*}
    Thus,
    \begin{equation*}
        p_k^{(n)} = [u^n z^k]F(u,z) = \frac{2^k}{(k-1)!} \frac{(\log n)^{k-1}}{n} + O\left(\frac{(\log n)^{k-2}}{n} + \frac{(\log n)^{k-1}}{n^2}\right).
    \end{equation*}
    
    \medskip
    For the second statement, we do an induction on $k\geq 0$ and show that
    \begin{equation*}
        p_{n-k}^{(n)} \overset{!}{=} p_n^{(n)} C_k(n),
    \end{equation*}
    where $C_k(n)$ is a polynomial in $n$ of degree $3k$ and such that
    \begin{equation*}
        C_k(n) = \frac{n^{3k}}{3^k k!} (1+O(1/n)),
    \end{equation*}
    for $n\to \infty$.
    By \eqref{eq:pnn} and \eqref{eqn:pnn1} the statement is true for $k=0$ with $C_0(n)=1$ and $k=1$ with $C_1(n) = \frac{1}{3}(n+1)n(n-1)$.
    Now assume that the statement is true for all $j<k$.
    We first observe that by \eqref{eq:pnn}, for $k\in \ZZ_{\geq 0}$,
    \begin{equation*}
        \binom{n+1}{2} p_{n-k}^{(n)} = p_{n-1-k}^{(n-1)} + 2 p_{n-1-k}^{(n-2)} + \ldots + (k+1) p_{n-1-k}^{(n-1-k)} = \sum_{j=1}^{k+1} j p_{n-k-1}^{(n-j)},
    \end{equation*}
    which follows by summing over the cases $i_{n-k}\in \{1,\dotsc,k\}$ and noting that $i_{n-k}>k$ is not possible.
    Thus, by the induction hypothesis on $k+1-j\in\{0,\dotsc,k-1\}$ and since
    \begin{equation*}
        \binom{\ell+1}{2} p_\ell^{(\ell)} = p_{\ell-1}^{(\ell-1)} \quad \text{for $\ell\in\ZZ_{\geq 2}$},
    \end{equation*}
    we derive
    \begin{align*}
        \binom{n+1}{2} p_{n-k}^{(n)}
        &= p_{(n-1)-k}^{(n-1)} + \sum_{j=2}^{k+1} j p_{(n-j)-(k+1-j)}^{(n-j)}\\
        &= p_{(n-1)-k}^{(n-1)} + \sum_{j=2}^{k+1} p_{n-j}^{(n-j)} j C_{k+1-j}(n-j)\\
        &= p_{(n-1)-k}^{(n-1)} + p_{n-1}^{(n-1)} \sum_{j=2}^{k+1} \left(\prod_{\ell=1}^{j-1} \binom{n-\ell+1}{2}\right) j C_{k+1-j}(n-j).
    \end{align*}
    Next, we note that
    \begin{align*}
        R_k(n):= \sum_{j=2}^{k+1} \left(\prod_{\ell=0}^{j-2} \binom{n-\ell}{2}\right) j C_{k+1-j}(n-j) 
        &= \sum_{j=2}^{k+1} \frac{j n^{2(j-1) + 3(k+1-j)}}{2^{j-1} 3^{k+1-j} (k+1-j)!} (1+O(1/n))\\
        &= n^{2k} \sum_{j=0}^{k-1} \frac{(k-j+1)}{2^{k-j}j!} \left(\frac{n}{3}\right)^{j} (1+O(1/n))\\
        &= \frac{n^{3k-1}}{3^{k-1} (k-1)!} (1+O(1/n)).
    \end{align*}
    We thus arrive at
    \begin{equation*}
        \binom{n+1}{2} p_{n-k}^{(n)} = p_{(n-1)-k}^{(n-1)} + p_{n-1}^{(n-1)} R_k(n),
    \end{equation*}
    where $R_k(n)$ is a polynomial in $n$ of degree $3k-1$.
    The solution for the homogeneous equation $\binom{n+1}{2} p_{n-k}^{(n)} = p_{n-1-k}^{(n-1)}$ is of the form $p_{n-k}^{(n)}=C p_{n}^{(n)}$ and therefore a general solution is of the form $p_{n-k}^{(n)} = \tilde{C}_k(n) p_{n}^{(n)}$. This yields
    \begin{equation*}
        \tilde{C}_k(n) = \tilde{C}_k(n-1) + R_k(n),
    \end{equation*}
    and we conclude
    \begin{align*}
        \tilde{C}_k(n) 
        &= \tilde{C}_k(1) + \sum_{j=0}^{n-2} R_k(n-j) = \frac{n^{3k-1}}{3^{k-1}(k-1)!} \sum_{j=0}^{n-2} \left(1-\frac{j}{n}\right)^{3k-1} (1+O(1/n)).
    \end{align*}
    Thus $\tilde{C}_k(n)$ is a polynomial of degree $3k$ and since, by an application of the Euler--Maclaurin formula, see \cite[Section 7.21]{WiWat},
    \begin{equation*}
        \sum_{j=0}^{n-2} \left(1-\frac{j}{n}\right)^{3k-1} =
        n \left(\int_{0}^{1} (1-x)^{3k-1}\, \dint x\right) + O(1) =
        \frac{n}{3k} (1+O(1/n)),
    \end{equation*}
    we conclude that $\tilde{C}_k(n) = \frac{n^{3k}}{3^k k!}(1+O(1/n))$.
    This completes the proof.
\end{proof}

\begin{remark}
    In the proof of \eqref{eqn:pnnk} we show that $p_{n-k}^{(n)}=\frac{2^n}{n! (n+1)!} C_k(n)$ where $C_k(n)$ is a polynomial of degree $3k$ that can be calculated precisely for fixed $k$. For example, for $k=2$ we find
    \begin{equation*}
        p_{n-2}^{(n)} 
        = \frac{2^n}{(n!)^2} \left(\frac{1}{18} n^5 - \frac{37}{180} n^4 + \frac{79}{360} n^3 - \frac{19}{360} n^2 - \frac{1}{60} n\right)
        = \frac{2^n n^5}{3^2 2! (n!)^2} (1+O(1/n)).
    \end{equation*}
\end{remark}

Note that Theorem \ref{thm:probab} gives
\begin{equation*}
    p_{k+1}^{(n)} = n^{-1+o(1)} \quad \text{and} \quad p_{n-k}^{(n)} = n^{-2n(1+o(1))},
\end{equation*}
for fixed $k\in\ZZ_{\geq 0}$. The error bounds in Theorem \ref{thm:probab} allow us to conclude that the same asymptotic rates hold for sequences $\{k_n\}_{n\in\NN}$ that grow below the error rate.

\begin{corollary}\label{cor:probab}
    For sequences $\{k_n\}_{n\in\NN}$ such that $k_n=o(\log n)$ we have
    \begin{equation*}
        p_{k_n+1}^{(n)} = \frac{2^{k_n+1}}{(k_n)!} \frac{(\log n)^{k_n}}{n} (1+o(1)) = n^{-1+o(1)},
    \end{equation*}
    and for sequences with $k_n=o(n)$ we have
    \begin{equation*}
        p_{n-k_n}^{(n)} = \frac{n^{3k_n-1}}{3^{k_n} (k_n)!} \frac{2^n}{(n!)^2} (1+o(1)) = n^{-2n(1+o(1))}.
    \end{equation*}
\end{corollary}

Finally, we deal with the sequence $\{k_n\}_{n\in\NN}$ whose growth rate is $O(\log n)$ where we provide two results. First, we may use the large deviation principle in Theorem \ref{thm:LDP} to find good asymptotic estimates, as described in the beginning of Section \ref{sec:MDPandLDP}.

\begin{proposition}
	Fix \(c>0\) and set \(k_n:=\lfloor c\log n\rfloor\). Then
	\[
	p_{k_n}^{(n)}=\PP\big(f_0(T_n)=k_n+2\big)=n^{-I(c)+o(1)}
	\]
	as $n\to\infty$, where $I$ is the rate function of the large deviation principle in Theorem \ref{thm:LDP}.
\end{proposition}
\begin{proof}
	To prove the result, we establish an upper bound directly from the large deviation principle in Theorem \ref{thm:LDP}. Afterwards, a corresponding lower bound is shown using the technique of exponential change of measure, a standard tool in large deviation theory also known as Esscher or Cram\'er transform, see \cite{DZlargeDeviations}. Also, previously established results following from the real-rootedness of the probability generating functions of the random variables $f_0(T_n)$, derived in \cite{GT:2021}, will be used in a crucial way. To simplify our notation, we define the random variables $X_n:=f_0(T_n)/\log n$, $n\in\NN$.
	
	\medspace
	
	\textit{Upper bound.}
	Fix \(\varepsilon\in(0,1)\). Then, for sufficiently large \(n\), we have $\frac{k_n+2}{\log n}\in[c-\varepsilon,c+\varepsilon]$, and hence
	\[
	p_{k_n}^{(n)}=\PP\Big(X_n=\frac{k_n+2}{\log n}\Big)
	\le \PP\big(X_n\in[c-\varepsilon,c+\varepsilon]\big).
	\]
	Applying Theorem \ref{thm:LDP} to the closed set $B=[c-\varepsilon,c+\varepsilon]$ we find that
	\[
	\limsup_{n\to\infty}\frac{1}{\log n}\log p_{k_n}^{(n)}
	\le -\inf_{x\in[c-\varepsilon,c+\varepsilon]} I(x).
	\]
	Since \(I\) is continuous on \([0,\infty)\),
	\(\inf_{x\in[c-\varepsilon,c+\varepsilon]} I(x)\to I(c)\) as \(\varepsilon\to 0^+\). Thus, letting \(\varepsilon\to 0^+\) yields
	\[
	\limsup_{n\to\infty}\frac{1}{\log n}\log p_{k_n}^{(n)}\le -I(c).
	\]
	
	\medskip
	\textit{Lower bound.}
	Recall the set-up of Section \ref{sec:StartingPoint}. For $c>0$ let $t=t(c)$ be the unique solution of $\mu'(t)=c$, or equivalently the maximizer of $tc-\mu(t)=I(c)$. Define, for sufficiently small $t$, the exponentially tilted measure $\PP_t^{(n)}$ whose Radon--Nikodym density with respect to the underlying probability measure $\PP$ equals $e^{t f_0(T_n)}/\EE e^{t f_0(T_n)}$.
	Then, for every integer $k> 0$,
	$$
		p_k^{(n)}=\PP(f_0(T_n)=k+2)
		=e^{-t (k+2)}\,\EE e^{t f_0(T_n)}\,\PP_t^{(n)}\big(f_0(T_n)=k+2\big).
	$$
	Applying this with $k=k_n$, dividing by $\log n$ and using \eqref{eq:GE-limit} gives
	\begin{align*}
		\frac{1}{\log n}\log p_{k_n}^{(n)}
		&=\frac{1}{\log n}\log \PP\big(f_0(T_n)=k_n+2\big)\\
		&=\mu(t)-t\,\frac{k_n}{\log n}
		+\frac{1}{\log n}\log \PP_t^{(n)}\big(f_0(T_n)=k_n+2\big)
		+\frac{\psi(t)}{\log n}+o(1).
	\end{align*}
	Since $k_n/\log n\to c$ we have $\mu(t)-t\,\tfrac{k_n}{\log n}=-I(c)+o(1)$, and it follows that
	\begin{equation}\label{eq:LB-main}
		\frac{1}{\log n}\log p_{k_n}^{(n)} = -I(c)+\Delta_n+o(1),
		\qquad
		\Delta_n:=\frac{1}{\log n}\log \PP_t^{(n)}\big(f_0(T_n)=k_n+2\big).
	\end{equation}
	
	To bound $\Delta_n$ from below, observe that under $\PP_t^{(n)}$ we have 
	\[
	\log \EE_t e^{s f_0(T_n)}
	=\log \frac{\EE e^{(t+s)f_0(T_n)}}{\EE e^{t f_0(T_n)}}
	=\big(\mu(t+s)-\mu(t)\big)\log n+\big(\psi(t+s)-\psi(t)\big)+o(1),
	\]
	uniformly for $s$ near $0$ from \eqref{eq:GE-limit}, where we write $\EE_t$ for the expectation (integration) with respect to $\PP_t^{(n)}$. In particular,
	\begin{equation}\label{eq:MeanExpTilt}
        \EE_t f_0(T_n) = \lim_{s\to 0} \frac{1}{s} \log \EE_t e^{s f_0(T_n)}=\mu'(t)\,\log n+\psi'(t)+o(1)=c\,\log n+\psi'(t)+o(1),
	\end{equation}
	and the sequence of random variables $X_n=f_0(T_n)/\log n$ satisfies under $\PP_t^{(n)}$ a large deviation principle with speed $\log n$ and a good rate function
	\[
	I_t(x):=I(x)+\mu(t)-t x = [I(x)-I(c)]-t(x-c).
	\]
	Note that by \eqref{eqn:LF-Irate} $I_t(x)\geq 0$ and it attains its unique minimum $0$ at $x=c$. Fix $\eta\in(0,1)$, define the closed set $B_\eta:=[c-\eta,c+\eta]$ and put
	\[
	\mathcal K_n(\eta):=\Big\{k\in\mathbb Z:\ \frac{k}{\log n}\in B_\eta\Big\}.
	\]
	By the large deviation principle for $X_n$ under the measure $\PP_t^{(n)}$, applied to the complement of $B_\eta$,
	\[
	\PP_t^{(n)}(X_n\notin B_\eta)
	= n^{-\inf_{x\notin B_\eta} I_t(x)+o(1)}
	= n^{-\delta_\eta+o(1)}
	\]
	with $\delta_\eta:=\inf_{x\notin B_\eta} I_t(x)>0$. Hence,
	\[
	\PP_t^{(n)}(X_n\in B_\eta)=1-o(1).
	\]
	Since $|\mathcal K_n(\eta)|=2\eta\log n+O(1)$, we deduce by averaging that
	\begin{equation}\label{eq:ZwischenschrittMax}
	\max_{k\in\mathcal K_n(\eta)} \PP_t^{(n)}(f_0(T_n)=k+2)
	\geq \frac{1-o(1)}{2\eta\log n+O(1)} = {1+o(1)\over 2\eta\,\log n}.
	\end{equation}

	On the other hand, it is a consequence of \cite[Theorem 6]{GT:2021} 
	that the sequence $\{p_k^{(n)}\}_{k=1}^n$ is ultra log-concave and hence unimodal. The same holds after the exponential change of measure.
	For such sequences, the mode lies within $O(1)$ of the mean, see \cite[page 123--124]{Stanley}. So, using \eqref{eq:MeanExpTilt},
	we conclude that the mode is at $c\log n+O(1)$. Since $k_n=\lfloor c\log n\rfloor$, it follows that $k_n$ is within
	$O(1)$ of the mode. In an ultra log-concave sequence of point probabilities, shifting
	$O(1)$ steps from the mode changes the probabilities by at most
	a constant factor. Indeed, if 
	$(a_k)$ is log-concave, then the ratios $R_k:=a_{k+1}/a_k$ form a 
	non-increasing sequence. In particular, around the mode $m$ we have 
	$R_{m-1}\ge 1 \ge R_m$, so each ratio in a fixed neighborhood of $m$ 
	is bounded away from $0$ and $\infty$. Hence, for any fixed $L\ge 0$,
	\[
	\frac{a_{m\pm L}}{a_m} =\prod_{i=0}^{L-1} R_{m+i}
	\qquad\text{or}\qquad 
	\frac{a_{m\pm L}}{a_m} = \prod_{i=1}^{L} \frac{1}{R_{m-i}},
	\]
	is bounded between two positive and finite constants depending only on $L$.
	
	Applying this to $a_k=\PP_t^{(n)}(f_0(T_n)=k+2)$ and combining it with \eqref{eq:ZwischenschrittMax} yields
	\[
	\PP_t^{(n)}(f_0(T_n)=k_n+2)
	\geq\frac{1+o(1)}{C\,\log n},
	\]
	for some constant $C>0$ independent of $n$. In particular,
	\[
	\Delta_n=\frac{1}{\log n}\log \PP_t^{(n)}\!\big(f_0(T_n)=k_n+2\big)\geq-\frac{\log\log n+O(1)}{\log n} \to 0.
	\]
	Plugging this lower bound for $\Delta_n$ back into \eqref{eq:LB-main} yields
	\[
	\liminf_{n\to\infty}\frac{1}{\log n}\log p_{k_n}^{(n)} \geq -I(c).
	\]
	Combining the lower with the upper bound completes the proof.
\end{proof}

%{\color{red}CT: Wenn wir straffen wollten, könnten wir das nächste Resultat nur als Bemerkung bringen und mit Hilfe des Bildes den Wert $3/4$ versuchen zu motivieren.}
%{\color{blue}FB: Ja, nachdem wir in Propostition 3.14 schon eine Asymptotik zeigen, könnten wir uns hier wohl leicht 2 Seiten sparen. Wir könnten hier nur eine Remark schreiben und vielleicht nur kurz erklären wie man zu dem Resultat kommt und das Bild als motivation für die Schranke von $c<3/4$. Ich denke ich kann mich auch damit anfreunden wenn Du denkst wir sollten es noch kürzen. Mir war das Resultat nur noch wichtig, damit man dann für $c=2/3$ die asymptotic von $p_{k_n}^{(n)}=O(1/\sqrt{\log n})$ sieht (die bei Proposition 3.14 untergeht weil $I(2/3)=0$).}

If one wants to have more precise information on the asymptotic behavior of $p_{k_n}^{(n)}$ as $n\to \infty$, in particular for the asymptotically maximal probability $p_{k_n}^{(n)}$ for $k_n=\lfloor\frac{2}{3}\log n\rfloor$, then the analysis will be more intricate. We provide the following result.

\begin{theorem}\label{thm:saddle}
    For $c\in(0,\frac{3}{4})$ and a sequence $\{k_n\}_{n\in\NN}$ defined by $k_n:=\lfloor c\log n\rfloor$ we have
    \begin{align*}
        p_{k_n}^{(n)} &= \frac{\sigma(\alpha_*)}{\sqrt{2\pi}} \left(\frac{1}{\alpha_*} + \frac{1}{\alpha_*+1}\right) \frac{K(\alpha_*)}{\Gamma(\alpha_*)} \binom{\alpha_*+1}{2}^{\!\!-k_n} 
            \frac{n^{\alpha_*-1}}{\sqrt{k_n}}\, (1+o(1)),
    \end{align*}
    where 
    \begin{equation*}
        \alpha_* := c-\frac{1}{2} + \frac{1}{2} \sqrt{4c^2+1}>0 \quad \text{and} \quad
        \sigma(\alpha_*) := \left(\frac{1}{\alpha_*^2} + \frac{1}{(\alpha_*+1)^2}\right)^{-1/2}>0.
    \end{equation*}
\end{theorem}
\begin{proof}
    As in the beginning of the proof of Theorem \ref{thm:probab} and using the Cauchy integral formula, we derive
    \begin{align*}
        p_{k_n}^{(n)} &= [z^{k_n}] G_n(z) \\
        &= [\alpha^{k_n}] \frac{1+2\alpha}{1+\alpha}\frac{K(\alpha)}{\Gamma(\alpha)} \left(\frac{2}{1+\alpha}\right)^{k_n} n^{\alpha-1} (1+O(1/n))\\
        &= \frac{1}{2\pi i} \oint \left(\frac{1}{\alpha}+\frac{1}{1+\alpha}\right)\frac{K(\alpha)}{\Gamma(\alpha)} \left(\frac{2}{\alpha(1+\alpha)}\right)^{k_n} 
             n^{\alpha-1} \,\dint \alpha\, (1+O(1/n))\\
        &= \frac{1}{2\pi i} \oint H(\alpha) \binom{\alpha+1}{2}^{-k_n}
            n^{\alpha-1} \, \dint\alpha\, (1+O(1/n)),
    \end{align*}
    where we set
    \begin{equation*}
        H(\alpha) := \left(\frac{1}{\alpha}+\frac{1}{\alpha+1}\right) \frac{K(\alpha)}{\Gamma(\alpha)} = \frac{1}{(1+\alpha)^2} \frac{\Gamma(2+2\alpha)}{\Gamma(1+\alpha)^3}.
    \end{equation*}
    The complex contour integral is over any simple loop $\alpha(t)$ in $\CC$ around $0$ with $\mathrm{Im}\, \alpha>-1$, since the integrand has a pole at $\alpha=-1$ and is analytic for $\mathrm{Im}\, \alpha>-1$ apart from the pole at $\alpha=0$ of order $k_n$.

By Cauchy's Theorem we may choose to consider the contour integral over a rectangle with vertices $-\sigma\pm i T$ and $b\pm T$ for some $\sigma\in (0,1)$ and large $T>0$.
We will show that the main contribution of the contour is given by the vertical contour $b+iT$, see Figure \ref{fig:graph_In}.

\begin{figure}[ht]
    \centering
    \begin{tikzpicture}
        \begin{scope}
            \clip (-2.5,-2.5) rectangle (2.5,2.5);
            \node at (0,0) {\includegraphics[width=5.5cm]{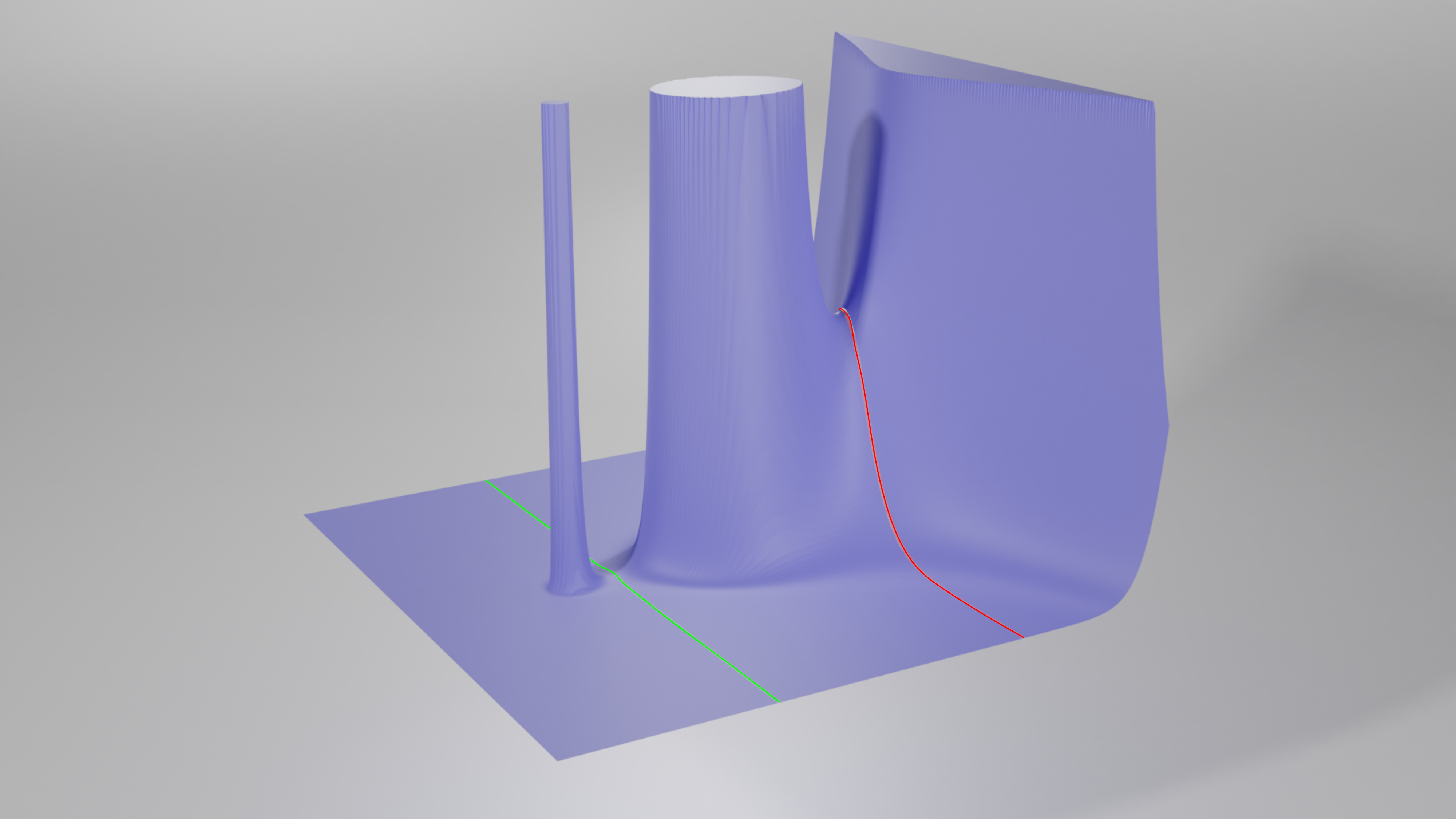}};
            %\draw[color=black!80, ->] (-0.65,-1.33) to (-2,0) node[left] {$it$};
            %\draw[color=black!80, ->] (-0.65,-1.33) to (2,-0.63) node[right] {$s$};
            %\draw (-2.5,-2.5) rectangle (2.5,2.5);
        \end{scope}
        
        \begin{scope}[xshift=6cm]
            \clip (-2.5,-2.5) rectangle (2.5,2.5);
            \node at (0,0) {\includegraphics[width=5.5cm]{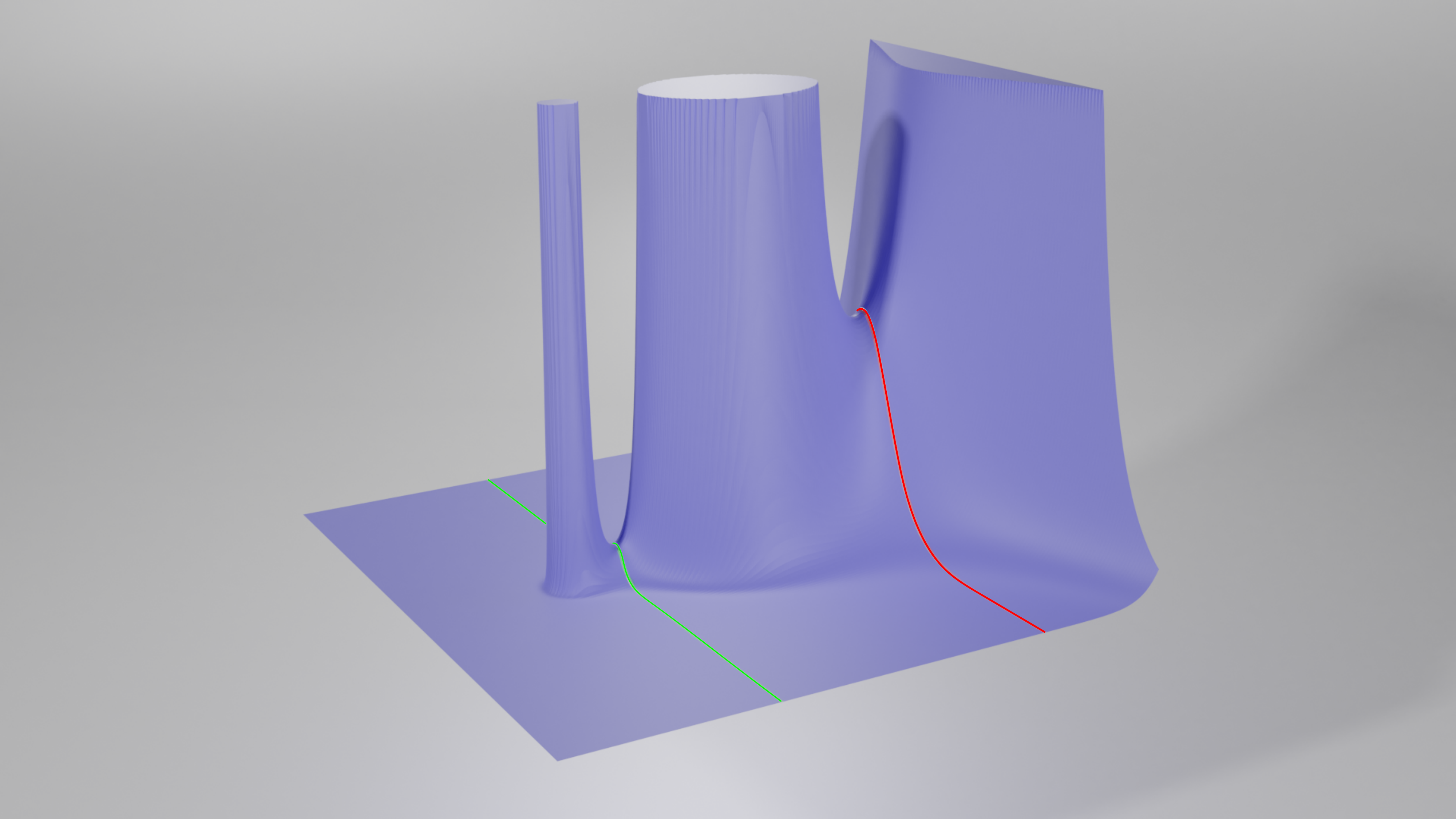}};
            %\draw (-2.5,-2.5) rectangle (2.5,2.5);
        \end{scope}
        
        \begin{scope}[xshift=12cm]
            \clip (-2.5,-2.5) rectangle (2.5,2.5);
            \node at (0.2,0) {\includegraphics[width=5.9cm]{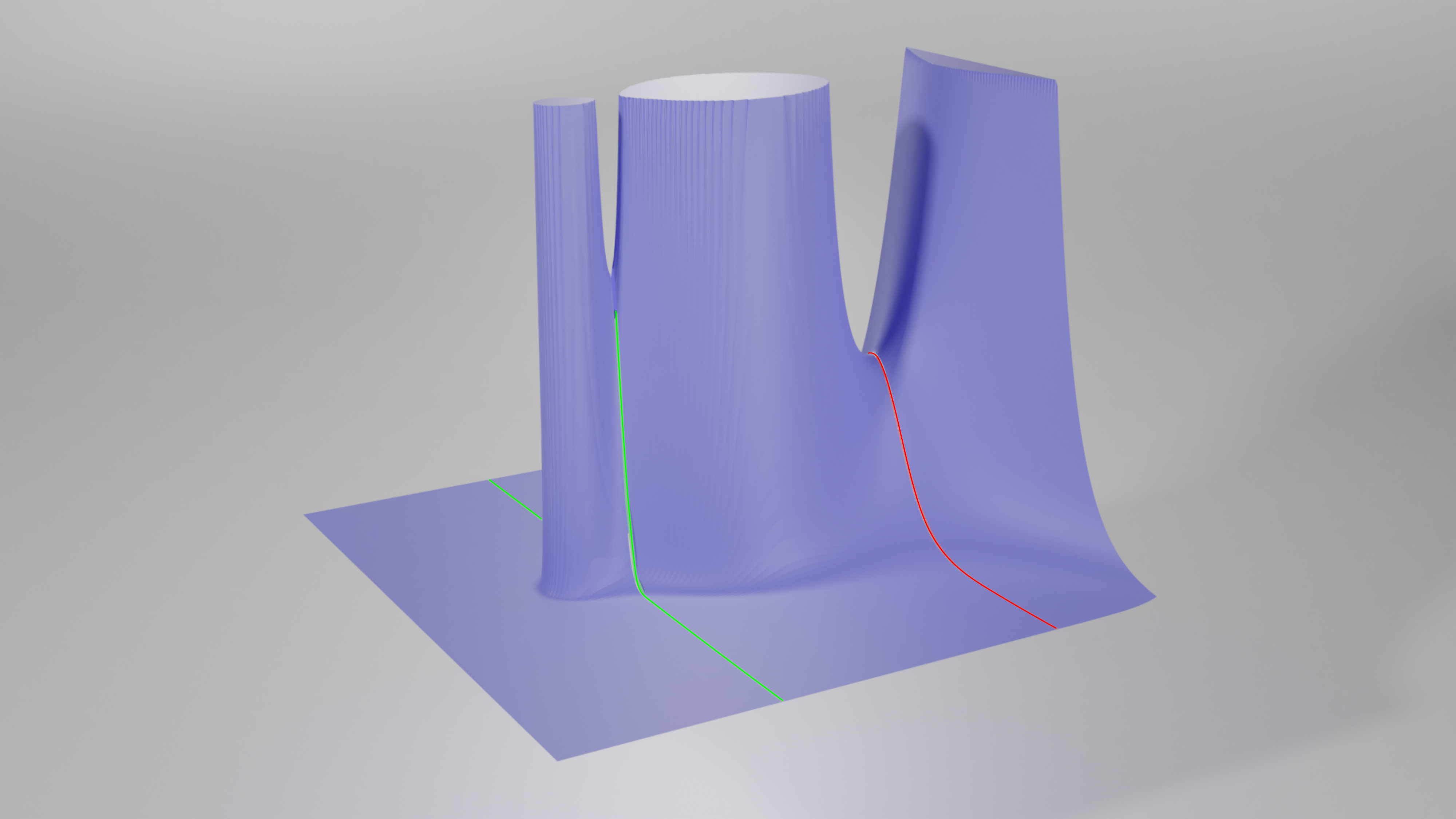}};
            %\draw (-2.5,-2.5) rectangle (2.5,2.5);
        \end{scope}
    \end{tikzpicture}
    \caption{Graph of $|I_n(s+it)|$ for $n=10^6$ and $c=2/3$ (left), $c=3/4$ (middle), and $c=0.8$ (right). The green curve is the section $s\equiv \tilde{\alpha}_*$ and the red curve is the section $s\equiv \alpha_*$.}
    \label{fig:graph_In}
    % Link to script:
    %https://sagecell.sagemath.org/?z=eJx9Uc1uozAQviPxDkg9MDaGYNNDVck9RqqqPkGUVA5xCA2xqXFbHr9jlmS726iWrPHw_WmYD-UgVWybkjgykpfTiaNalsVdHCmeyKTO-UJkeOnw5jzc0nqDLfKVOKP5FfRo5L6z1kFNO9uAIfitd63xcDT4fEKt2g7AF8CzkWwEbdTppEDQMRNk8aeZkCqhidnAmHNCAUY0p-NCkA3kx8k0jnZ6nywBpyD3cZTgcdqHcOXhqRjeMWWUKnuk20APeLsPlIdqps-Sd2eSqrg4hBbLJYDDd_-AYibH0AtDXGOImXGefhKVf__GRCnJzLnpO-urHSzZSjHF87LgWDIsa7baspwzfNS2s06m1inT6LC5vvPPWv4jxZ0E4SwTTKxZwF96i6GDXOGqGd71WZ3JXjl10t619cvstArjseV3jzna6V3K_KGtj0YPg6x-dxHBRfx0aZzW5n-f2akYDvYTemdfde1ba8K4_mAbp3qkp-QLdw6_Zg==&lang=sage&interacts=eJyLjgUAARUAuQ==
\end{figure}
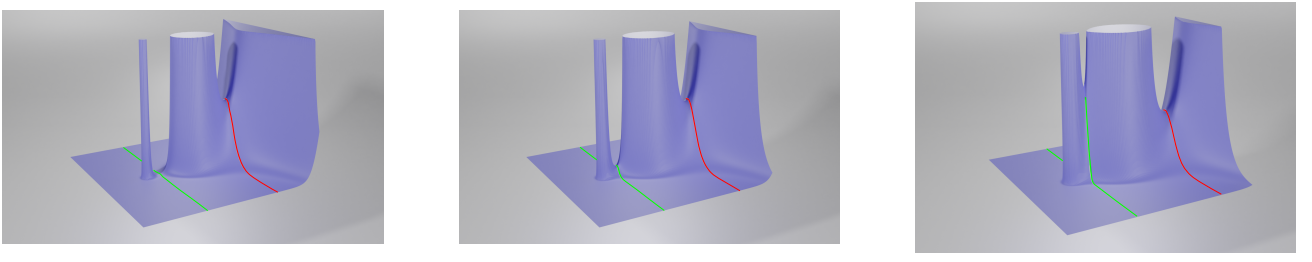

\medskip 
\emph{Estimating the horizontal contours:} 
By Stirling's formula we may estimate for all $s>0$,
\begin{align*}
    |\Gamma(s+iT)| &= e^{-\frac{\pi}{2} |T| (1+o(1))} \quad \text{as $|T| \to \infty$}
\end{align*}
which yields for all $s\in [-\sigma,b]\subset (-1,+\infty)$
\begin{equation*}
    |H(s+iT)| \leq e^{C_1|T|},
\end{equation*}
for some $C_1=C_1(\sigma,b)>0$ and $|T|$ large enough.
Furthermore,
\begin{equation*}
    \left|\binom{1+s+iT}{2}^{-k_n}\right| 
    \leq n^{-C_2|T|},
\end{equation*}
for some $C_2=C_2(\sigma,b,c)>0$, and
\begin{equation*}
    |n^{s+iT-1}| = n^{s-1}\leq n^{b-1}.
\end{equation*}
Writing
\begin{equation*}
    I_n(\alpha) := H(\alpha)\, \binom{\alpha+1}{2}^{-k_n}\, n^{\alpha-1}, 
    %= H(\alpha) n^{\alpha - 1 - c\log\binom{\alpha+1}{2} + O(1/\log n)},
\end{equation*}
we conclude
\begin{equation*}
    \left|\int_{-\sigma}^b I_n(s+iT)\, \dint s \right| \leq (b+\sigma) n^{C_1\frac{|T|}{\log n}-C_2|T|+b-1}.
\end{equation*}
So for $T=T_n:=n^2$ we find
\begin{equation*}
    \left|\frac{1}{2\pi i} \int_{-\sigma}^b I_n(s+iT_n)\, \dint s\, (1+O(1/n)) \right| \leq n^{-C_3 n^2},
\end{equation*}
for some $C_3=C_3(c,b,\sigma)>0$.
Thus, the horizontal parts of our rectangle will be asymptotically small.
\medskip

\emph{Estimating the vertical contours:}
Next, we want to estimate the vertical parts
\begin{equation*}
    \frac{1}{2\pi i}\int_{\mathrm{Re}\, \alpha=s_0, |\mathrm{Im}\,\alpha|\leq T_n} I_n(\alpha)\, \dint \alpha\, (1+O(1/n)),
\end{equation*}
for suitable fixed $s_0\in (-1,+\infty)$.
For this, we write
    \begin{align*}
        I_n(\alpha)&= n^{S_n(\alpha)}
    \end{align*}
where we set
\begin{equation*}
    S_n(\alpha) := \alpha - 1 + \frac{\log H(\alpha)}{\log n}-(\log \alpha) \frac{k_n}{\log n} - \log\left(\frac{1+\alpha}{2}\right)\frac{k_n}{\log n}.
\end{equation*}
Then
\begin{equation*}
    S_n'(\alpha) = 1 + \frac{H'(\alpha)}{H(\alpha)\log n} - \frac{1}{\alpha} \frac{k_n}{\log n} - \frac{1}{1+\alpha} \frac{k_n}{\log n} + O\left(\frac{1}{\log n}\right).
\end{equation*}
We set
\begin{equation*}
    b=\alpha_* := c-\frac{1}{2}+\frac{1}{2}\sqrt{4c^2+1} >0 \quad \text{and}\quad
    -\sigma=\tilde{\alpha}_* := c-\frac{1}{2}-\frac{1}{2}\sqrt{4c^2+1} \in (-1,-1/2),
\end{equation*}
and conclude that for $\alpha^\dag\in\{\alpha_*,\tilde{\alpha}_*\}$
\begin{equation*}
    \lim_{n\to\infty} S_n'(\alpha^\dag) = 1-c\left(\frac{1}{\alpha^\dag} +\frac{1}{1+\alpha^\dag}\right)=0,
\end{equation*}
which in turn yields
    \begin{equation*}
        S_n'(\alpha^\dag) = \frac{H'(\alpha^\dag)}{H(\alpha^\dag)\log n} + \left(\frac{1}{\alpha^\dag} +\frac{1}{\alpha^\dag+1}\right) \frac{c\log n-\lfloor c\log n\rfloor}{\log n} 
             + O\left(\frac{1}{\log n}\right) 
            = O\left(\frac{1}{\log n}\right).
    \end{equation*}
    Furthermore,
    \begin{equation*}
        S''_n(\alpha^\dag) = \left(\frac{1}{(\alpha^\dag)^2} + \frac{1}{(1+\alpha^\dag)^2}\right) \frac{k_n}{\log n} + O\left(\frac{1}{\log n}\right) 
        = \frac{c}{\sigma(\alpha^\dag)^2} \left(1+ O\left(\frac{1}{\log n}\right)\right).
    \end{equation*}
    It follows that
    \begin{equation*}
        S_n(\alpha) = S_n(\alpha^\dag) + \frac{1}{2} S''_n(\alpha^\dag) (\alpha-\alpha^\dag)^2 + O\left(\frac{\alpha-\alpha^\dag}{\log n}\right) + O((\alpha-\alpha^\dag)^3)
    \end{equation*}
    So on the vertical contours $\mathrm{Re}\,\alpha=\alpha^\dag$ we have
    \begin{equation*}
        S_n(\alpha^\dag+it) = S_n(\alpha^\dag) - \frac{t^2}{2} S''_n(\alpha^\dag) + iO\left(\frac{t}{\log n}+t^3\right),
    \end{equation*}
    and therefore
    \begin{align*}
            &\frac{1}{2\pi} \int_{-T_n}^{T_n} I_n(\alpha^\dag+it) \, \dint t\, (1+O(1/n))\\
            &\qquad=\frac{1}{2\pi} \int_{-T_n}^{T_n} n^{S_n(\alpha^\dag)-\frac{t^2}{2} S''_n(\alpha^\dag)} e^{i O(t+(\log n)t^3)} \, \dint t\, (1+O(1/n))\\
            &\qquad= \frac{n^{S(\alpha^\dag)}}{2\pi \sqrt{(\log n)S''_n(\alpha^\dag)}} \int_{-T_n\sqrt{(\log n)S''_n(\alpha^\dag)}}^{T_n\sqrt{(\log n)S''_n(\alpha^\dag)}} e^{-s^2/2} e^{i O\left(\frac{s}{\sqrt{\log n}}\right)} \,\dint s\, (1+O(1/\sqrt{\log n}))\\
            &\qquad= \frac{\sigma(\alpha^\dag)}{2\pi} \frac{n^{S(\alpha^\dag)}}{\sqrt{k_n}} \int_{-\infty}^{+\infty} e^{-s^2/2} \left(1+O\left(\frac{s}{\sqrt{\log n}}\right)\right)\,\dint s\, (1+O(1/\sqrt{\log n}))\\
            &\qquad= \frac{\sigma(\alpha^\dag)}{\sqrt{2\pi}} \frac{n^{ S_n(\alpha^\dag)}}{\sqrt{k_n}} (1+O(1/\sqrt{\log n})).
    \end{align*}
    For $\alpha^\dag=\tilde{\alpha}_*$ we further estimate
    \begin{equation*}
        \left|\frac{1}{2\pi} \int_{-T_n}^{T_n} I_n(\tilde{\alpha}_*+it) \, \dint t\, (1+O(1/n)) \right|
        \leq \frac{C_4}{\sqrt{k_n}} n^{\tilde{\alpha}_*-1-c \log|\binom{\tilde{\alpha}_*+1}{2}|},
    \end{equation*}
    and conclude
    \begin{align*}
        p_{k_n}^{(n)} 
        &= \frac{1}{2\pi} \int_{-T_n}^{T_n} I_n(\alpha_*+it) \, \dint t (1+O(1/n)) 
            +  \frac{C_4}{\sqrt{k_n}} n^{\tilde{\alpha}_*-1-c \log|\binom{\tilde{\alpha}_*+1}{2}|} + n^{-C_3n^2}\\
        &=\frac{\sigma(\alpha_*)}{\sqrt{2\pi}} \frac{H(\alpha_*)}{\sqrt{k_n}}n^{\alpha_*-1-c\log\binom{\alpha_*+1}{2}} (1+O(1/\sqrt{\log n}))
            +  \frac{C_4}{\sqrt{k_n}} n^{\tilde{\alpha}_*-1-c \log|\binom{\tilde{\alpha}_*+1}{2}|} + n^{-C_3n^2}.
    \end{align*}
    This completes the proof, since we can check that 
    \begin{equation*}
        -I(c) = \alpha_*-1-c\log\binom{\alpha_*+1}{2} > \tilde{\alpha}_*-1-c \log\left|\binom{\tilde{\alpha}_*+1}{2}\right|,
    \end{equation*}
    for $c\in (0,c_0)$, where we choose $c_0:=\frac{3}{4}$ somewhat arbitrarily.
\end{proof}

\begin{remark}
Theorem \ref{thm:saddle} gives
\begin{equation*}
    p_{k_n}^{(n)} =
        C\frac{n^{-I(c)}}{\sqrt{\log n}} (1+o(1)) 
    = n^{-I(c)+o(1)},
\end{equation*}
where $C=C(c)>0$ is an explicit constant and $I$ is again the rate function from the large deviation principle in Theorem \ref{thm:LDP}.
Notice that
\begin{equation*}
    \lim_{c\to 0^+} -I(c) = -1,
\end{equation*}
and for $c_n:=\frac{n}{\log n}$,
\begin{equation*}
    -I(c_n) = -2n (1+o(1)).
\end{equation*}
Thus the boundary cases of Theorem \ref{thm:saddle} connect with the asymptotic rate derived for $p_{k_n}^{(n)}$ in Corollary \ref{cor:probab}.
For the unique minimum of $I(c_*)=0$ at $c_*=2/3$ we derive for $k_n=\lfloor c_*\log n\rfloor$ that
$\alpha_*=1$, and therefore
\begin{equation*}
    p_{k_n}^{(n)} = \sqrt{\frac{27}{10}} \frac{1}{\sqrt{2\pi \log n}} (1+o(1)).
\end{equation*}
\end{remark}

\subsection*{Acknowledgement}
CT has been supported by the DFG through SPP 2458 \textit{Combinatorial Synergies}.

\bibliographystyle{plain}
\bibliography{References}
%\printbibliography

\end{document}